\documentclass[11pt]{amsart}

\usepackage{amssymb,amsmath,latexsym}
\usepackage{amsthm}
\usepackage{enumerate}
\usepackage{booktabs}
\usepackage{mathrsfs}

\usepackage{hyperref}
\usepackage{xcolor}
\hypersetup{
    colorlinks, linkcolor={red!80!black},
    citecolor={blue!80!black}, urlcolor={blue}
}

\usepackage{graphicx,pgf}
\usepackage{float}

\usepackage{caption}
\usepackage{subcaption}

\usepackage{fancyhdr}

\newtheorem{thm}{Theorem}[section]
\newtheorem{cor}{Corollary}[section]
\newtheorem{defn}{Definition}[section]
\newtheorem{prop}{Proposition}[section]
\newtheorem{lemma}{Lemma}[section]
\newtheorem{rmk}{Remark}[section]

\newcommand{\aiminabs}[1]{ | #1 |}
\newcommand{\aiminnorm}[1]{\big\lVert#1\big\rVert}
\newcommand{\aimininner}[2]{\langle #1, #2 \rangle}

\newcommand{\aiminRmnum}[1]{ \uppercase\expandafter{\romannumeral  #1}}

\numberwithin{equation}{section}
\numberwithin{figure}{section}

\begin{document}
\title{The nonlinear 2D supercritical inviscid shallow water equations in a rectangle}
\author{Aimin Huang}
\author{Madalina Petcu}
\author{Roger Temam}
\address{The Institute for Scientific Computing and Applied Mathematics, Indiana University, 831 East Third Street, Rawles Hall, Bloomington, Indiana 47405, U.S.A.}
\address{Laboratoire de Math\'ematiques et applications, Univ. de Poitiers, Teleport 2-BP 30179, Boulevard Marie et Pierre Curie, 86962 Futuroscope Chasseneuil Cedex, France}
\email{AH:huangepn@gmail.com}
\email{MP:Madalina.Petcu@math.univ-poitiers.fr}
\email{RT:temam@indiana.edu}
\date{\today}

\begin{abstract}

In this article we consider the inviscid two-dimensional shallow water equations in a rectangle. The flow occurs near a stationary solution in the so called supercritical regime and we establish short term existence of smooth solutions for the corresponding initial and boundary value problem.
\end{abstract}

\maketitle
\section{Introduction}\label{sec-intro}
Motivated by the study of the inviscid primitive equations, we consider in this article the inviscid two-dimensional shallow water equations in a rectangle in the so-called supercritical regime. It has been shown that a certain vertical expansion of the inviscid primitive equations leads to a system of coupled nonlinear equations similar to the inviscid shallow water equations; see \cite{RTT08b} and \cite{HT14a}. Hence beside their intrinsic interest, the nonlinear shallow water equations can be seen as one mode of the vertical expansion of the primitive equations.

The issue of the boundary conditions to be associated with the primitive or shallow water equations has been emphasized as a major problem and limitation for the so-called Local Area Models for which weather predictions are sought and simulations are performed within a domain for which the boundary has no physical significance, so that there are no physical laws prescribing the boundary conditions (see \cite{WPT97} and e.g. \cite{RTT08a,RTT08b}, \cite{CSTT12,SLTT}).
The choice of the boundary conditions relies then on mathematical considerations
(derivation of a well-posed mathematical problem), and on general computational considerations and physical intuition. The boundary conditions suitable for the one-dimensional shallow water equations were derived in an intuitive context in the book of Whitham \cite{Whi99} and in \cite{NHF08}; see \cite{PT13} for a rigorous study. For general results on boundary value problems for quasilinear hyperbolic system in space dimension one see \cite{LY85}; for initial and boundary value problems for hyperbolic equations in smooth domain see the thorough book \cite{BS07}. The present article follows the study of the one-dimensional inviscid shallow water equations in \cite{PT13,HPT11} and the study of the linearized shallow water equations in \cite{HT14a}. In the study of the linearized inviscid shallow water equations in \cite{HT14a} we have shown that five cases can occur depending on the respective values of the velocity and the height (not counting the non-generic cases and the symmetries). The nonlinear case that we consider in this article relates to what was called the supercritical case in \cite{HT14a}; see \cite{HT14b} for the study of a subcritical case.

In this article, we consider the inviscid fully nonlinear 2D shallow water equations (SWE) 
\begin{equation}\label{eq1.1}
\begin{cases}
u_t+uu_x + vu_y + g\phi_x -fv= 0, \\
v_t+uv_x + vv_y + g\phi_y + fu= 0, \\
\phi_t+u\phi_x + v\phi_y + \phi(u_x+v_y) = 0;
\end{cases}
\end{equation}
here $U=(u,v,\phi)^t,(x,y)\in \Omega:=(0,L_1)\times(0,L_2), t\in (0, T)$, $u$ and $v$ are the two horizontal components of the velocity, $\phi$ is the height of the water, and $g$ is the gravitational acceleration, $f$ is the Coriolis parameter.
The first and second equations \eqref{eq1.1} are derived from the equations of conservation of horizontal momentum, and the third one expresses the conservation of mass.
We consider equations \eqref{eq1.1} for certain values of $u,v,\phi$ as described below, corresponding to a "supercritical" flow and we associate with \eqref{eq1.1}, initial conditions for $u,v,\phi$ and boundary conditions at $x=0$ and $y=0$, $u,v,\phi$ vanishing on that part of the boundary.

This article is organized as follows. After this introductory section, we derive in Section \ref{sec-density} suitable density theorems, density of certain smooth functions in certain function spaces of Sobolev type. Section \ref{sec-swe-operator} is devoted to the modified (symmetrized) SWE operator for the time-independent and the time-dependent cases. It prepares Section \ref{sec-linear} in which we deal with the linear SWE, linearized around a non-constant time-dependent flow unlike in \cite{HT14a} where the background flow is time-independent. In Section \ref{sec-linear} we prove the well-posedness of the linearized SWE at the price of a loss of derivatives (see Theorem \ref{thm6.1}), and then the well-posedness of the linearized SWE in a short time without a loss of derivatives (see Theorem \ref{thm6.2}). Section \ref{sec-nonlinear-swe} considers the fully nonlinear SWE, for which the local well-posedness result is obtained. In the Appendices \ref{sec-semigroup} and \ref{sec-classical-lemmas}, we collect some useful theorems about semigroup and evolution systems, and several classical estimates about functions in Sobolev spaces. 

%
%

\section{The density theorems}\label{sec-density}
In this section, we establish general density theorems for certain Sobolev spaces; the results supplement and complement those of Section 3 in \cite{HT14a} which we recall when needed. These theorems have independent interest, and also will be needed for proving later on that $-A$ generates a quasi-contraction semigroup on certain Sobolev spaces, where $A$ is the 2D modified SWE operator  associated with suitable boundary conditions.

Throughout this section, let $m$ be a non-negative integer and let $\lambda=\lambda(x,y)$ satisfy 
\begin{equation}\label{asp3.1}
c_0 \leq \lambda(x,y) \leq c_1, \\
\end{equation}
where $c_0,c_1$ are positive constants. Furthermore, we say that $\lambda=\lambda(x,y)$ satisfies the positive $m$-condition ($m\text{ integer }\geq 0$) if $\lambda$ satisfies \eqref{asp3.1} and 
\begin{equation}\label{asp3.2}
\begin{cases}
\nabla \lambda \in L^\infty(\Omega),\,\,\,\text{for }m=0,1,\\
\lambda\in H^{3 \vee m}(\Omega),\,\,\text{for }m\geq 2,\\
\end{cases}
\end{equation}
where $a\vee b=\text{max}(a,b)$.
It is easy to see that if $m_1\geq m_2$ and $\lambda$ satisfies the positive $m_1$-condition, then $\lambda$ also satisfies the positive $m_2$-condition.

We now set for any function $\theta=\theta(x,y)$, $T\theta=\lambda(x,y) \theta_x + \theta_y$, where $\lambda$ is assumed to satisfy the positive $m$-condition, and introduce the function space
\begin{equation}\nonumber
\mathcal{X}_1^m(\Omega)=\{ \theta\in H^m(\Omega)\,:\, T\theta=\lambda \theta_x + \theta_y\in H^m(\Omega) \}.
\end{equation}

We observe that $\mathcal{X}_1^m(\Omega)$ is a space of local type, that is
\begin{equation}\label{eq3.e1}
 \text{If }\theta\in \mathcal{X}_1^m(\Omega), \psi\in\mathcal{C}^\infty(\overline\Omega), \text{ then } \theta\psi\in\mathcal{X}_1^m(\Omega).
\end{equation}
This property follows from $T(\psi\theta)=\psi T\theta+(\lambda\psi_x+\psi_y)\theta$, and $\lambda\theta$ (and hence $(\lambda\psi_x+\psi_y)\theta$) is in $H^m(\Omega)$ because of Lemma \ref{lemb.1} $i)$.
\\

We give an equivalent characterization of the space $\mathcal{X}_1^m(\Omega)$. In the following and throughout this article, we let $\partial^\alpha=\partial_x^{\alpha_1}\partial_y^{\alpha_2}$ with $\alpha=(\alpha_1,\alpha_2)$ and set $\aiminabs{\alpha}=\alpha_1+\alpha_2$. We also denote by $[\partial^\alpha, f]$ the commutator $[\partial^\alpha, f]g=\partial^\alpha(fg)-f\partial^\alpha g$.
\begin{prop}\label{prop3.1}
We assume that $\lambda$ satisfies the positive $m$-condition. Then
\begin{equation}\label{eq3a.1a}
\mathcal{X}_1^m(\Omega)=\{ \theta\in H^m(\Omega)\,:\, T\partial^\alpha\theta\in L^2(\Omega),\,\forall\, \aiminabs{\alpha}=m \}.
\end{equation}
\end{prop}
\begin{proof}
It is clear that \eqref{eq3a.1a} holds when $m=0$. For $m=1$, we observe that
\begin{equation}\label{eq3.1.1}
T\partial^\alpha\theta = \partial^\alpha(T\theta) - (\partial^\alpha \lambda) \theta_x,
\end{equation}
where $\partial^\alpha=\partial_x$ or $\partial_y$. Then if $\theta\in H^1(\Omega)$, $(\partial^\alpha\lambda)\theta_x$ belongs to $L^2(\Omega)$ since $\nabla\lambda$ is bounded, and by \eqref{eq3.1.1}, $T\partial^\alpha\theta$ belongs to $L^2(\Omega)$ if and only if $\partial^\alpha(T\theta)$ belongs to $L^2(\Omega)$; \eqref{eq3a.1a} follows for $m=1$.

For $m\geq 2$, we observe that
\begin{equation}\label{eq3.1.2}
T\partial^\alpha\theta =\partial^\alpha(T\theta) - [\partial^\alpha, \lambda]\theta_x,
\end{equation}
holds for all $\aiminabs{\alpha}=m$. We note that for $\theta\in H^m(\Omega)$, $[\partial^\alpha,\lambda]\theta_x$ belongs to $L^2(\Omega)$ from Lemma \ref{lemb.1} $\ref{en:item2})$ with $d=2$, $a=\lambda$ and $k=3\vee m$.
Hence by \eqref{eq3.1.2}, for $\aiminabs{\alpha}=m$, $T\partial^\alpha\theta$ belongs to $L^2(\Omega)$ if and only if $\partial^\alpha(T\theta)$ belongs to $L^2(\Omega)$, and \eqref{eq3.1} follows as well for $m\geq 2$.

\end{proof}

Now, we need to show that the smooth functions are dense in $\mathcal{X}_1^m(\Omega)$. Later on we will prove more involved density theorems, showing that if $u\in\mathcal{X}_1^m(\Omega)$ vanishes on certain parts of $\partial\Omega$, then $u$ can be approximated in $\mathcal{X}_1^m(\Omega)$ by smooth functions, vanishing on the same parts of the boundary. 
For the moment, we prove the following:
\begin{prop}\label{prop3.2}
$\mathcal{C}^\infty(\overline\Omega)\cap\mathcal{X}_1^m(\Omega)$ is dense in $\mathcal{X}_1^m(\Omega)$.
\end{prop}
\begin{proof}
Using a proper covering of $\Omega$ by sets $\mathcal{O}_0,\mathcal{O}_1,\cdots,\mathcal{O}_N$, we consider a partition of unity subordinated to this covering, $1=\sum_{i=0}^N \psi_i$. Here and again in this section we will use a covering of $\Omega$ consisting of $\mathcal{O}_0$, a relatively compact subset of $\Omega$, and of sets $\mathcal{O}_i$ of one of the following types: $\mathcal{O}_i$ is a ball centered at one of the corners of $\Omega$, which does not intersect the two other sides of $\Omega$; or $\mathcal{O}_i$ is a ball centered on one of the sides of $\Omega$ which does not intersect any of the three other sides of $\Omega$.

If $\theta\in\mathcal{X}_1^m(\Omega)$, then $\theta\psi_i\in\mathcal{X}_1^m(\Omega)$ by \eqref{eq3.e1}, so that we only need to approximate $\theta\psi_i$ by smooth functions. Here the support of $\psi_i$ is contained in the set $\mathcal{O}_i$, and we start with considering the set $\mathcal{O}_0$, relatively compact in $\Omega$, then we consider the balls $\mathcal{O}_i$ centered on the boundary $\partial\Omega$.

For any function $v$ defined on $\Omega$, here and again in the following we denote by $\tilde v$ the function equal to $v$ in $\Omega$ and to $0$ in $\mathbb{R}^2\backslash\Omega$. 
We first consider the case $\psi_i=\psi_0$ and $\mathcal{O}_i =\mathcal{O}_0$ which is relatively compact in $\Omega$.
Let $\rho$ be a mollifier such that $\rho\geq 0, \int\rho=1$, and $\rho$ has compact support. 

$i)$ The function $v=\theta\psi_0\in\mathcal{X}_1^m(\Omega)$ has compact support in $\mathcal{O}_0$. Since $\mathcal{O}_0$ is relatively compact in $\Omega$, then for $\epsilon$ small enough, $\rho_\epsilon*v$ is supported in $\Omega$. Using the characterization \eqref{eq3.1} for $v$, the standard mollifier theory (see e.g. Appendix C in \cite{Eva98}) shows that for $\epsilon\rightarrow 0$:
\begin{equation}\label{eq05}
\begin{cases}
\rho_\epsilon* v \rightarrow v,\hspace{6pt} \text{in }H^m(\Omega),\\
\rho_\epsilon*T\partial^\alpha v\rightarrow T\partial^\alpha v,\hspace{6pt} \text{in }L^2(\Omega),\,\forall\, \aiminabs{\alpha}=m.
\end{cases}
\end{equation}
Since the convolution and the operator $T$ do not commute in the non-constant coefficient case, we need the following Friedrichs' lemma (see \cite{Fri44} or \cite[Theorem 3.1]{Hor61}).
\begin{lemma}\label{lem3.2extra}
Let ${\mathcal U}$ be an open set of $\mathbb{R}^d$.
If $\nabla a\in L^\infty(\mathcal U)$ and $u\in L_{\text{loc}}^2(\mathcal U)$, then for all $1\leq j\leq d$,
\begin{equation}\nonumber
a \partial_{x_j}(u*\rho_\epsilon) - (a\partial_{x_j} u)*\rho_\epsilon \rightarrow 0, \text{ when }\epsilon\rightarrow 0,
\end{equation}
in the sense of $L^2$ convergence on all compact subsets of $\mathcal{U}$.
\end{lemma}
We then continue the proof of Proposition \ref{prop3.2}. 
Noting that $v=\theta\psi_0$ has compact support in $\Omega$, we apply Lemma \ref{lem3.2extra} with $\mathcal{U}=\Omega, a=\lambda$ and $u=\partial^\alpha v$; we obtain 
\begin{equation}\label{eq06extra2}
T(\rho_\epsilon*\partial^\alpha v) - \rho_\epsilon*T\partial^\alpha v \rightarrow 0,\,\forall\, \aiminabs{\alpha}=m,
\end{equation}
in $L^2(\Omega)$ as $\epsilon\rightarrow 0$.
Combining \eqref{eq05}$_2$ and \eqref{eq06extra2}, we obtain that as $\epsilon\rightarrow 0$,
\begin{equation}\label{eq06}
T(\rho_\epsilon*\partial^\alpha v) \rightarrow T\partial^\alpha v,\,\text{in }L^2(\Omega),\,\forall\, \aiminabs{\alpha}=m.
\end{equation}
Therefore, $\rho_\epsilon*v$ converges to $v$ in $\mathcal{X}_1^m(\Omega)$ by \eqref{eq3a.1a}, $\eqref{eq05}_1$ and \eqref{eq06}.

$ii)$ We then consider the case where $\psi_i=\psi_1$, and $\mathcal{O}_i =\mathcal{O}_1$ which is a ball centered at the origin $(0,0)$; the other cases are similar or simpler. Set $v=\theta\psi_1$, and note that $v$ does not vanish in general on the boundary $\partial\Omega$ of $\Omega$. In order to extend $v$ to the whole space $\mathbb{R}^2$, we use a well known extension result (see e.g. \cite[Theorem 1.4.3.1]{Gri85}).
\begin{lemma}[Extension Theorem]\label{lem3.2}
Since the boundary $\partial\Omega$ of the domain $\Omega$ is Lipschitz continuous, there exists a continuous linear operator $P=P_m$ from $H^m(\Omega)$ into $H^m(\mathbb{R}^2)$ such that for all $u\in H^m(\Omega)$, the restriction of $Pu$ to $\Omega$ is $u$ itself, i.e.
\begin{equation}\nonumber
Pu|_\Omega = u.
\end{equation}
\end{lemma}

We denote by $\hat v$ the extension $Pv$ given in Lemma \ref{lem3.2}, and then observe that, for all $\aiminabs{\alpha}=m$,
\begin{equation}\label{eq06-1}
\begin{cases}
\partial^\alpha\hat v = \widetilde{\partial^\alpha v} + \nu_1^\alpha, \\
T\widetilde{\partial^\alpha v} =\widetilde{T\partial^\alpha v} + \nu_2^\alpha,
\end{cases}
\end{equation}
where $\nu_1^\alpha$\,\footnote[1]{In fact, $\nu_1^\alpha$ is the sum of a function with support in $\mathcal{O}_1\setminus\overline\Omega$ and a measure supported by $\mathcal{O}_1\cap\partial\Omega$, but this additional information is not useful to us.} is a measure supported by $\mathcal{O}_1\setminus\Omega$, and $\nu_2^\alpha$ is a measure supported by $\mathcal{O}_1\cap\partial\Omega$.

The two identities in \eqref{eq06-1} together show that
\begin{equation}\label{eq06-7}
T\partial^\alpha\hat v = \widetilde{T\partial^\alpha v} + \mu^\alpha,\, \forall\, \aiminabs{\alpha}=m,
\end{equation}
where the $\mu^\alpha$ are measures supported by $\mathcal{O}_1\setminus\Omega$. Let $\rho$ be the same mollifier as before, but now $\rho$ is compactly supported in $\{x<0, y<0\}$; then mollifying \eqref{eq06-7} with this $\rho$ gives
\begin{equation}\label{eq07}
\rho_\epsilon*(T\partial^\alpha\hat v) =\rho_\epsilon*\widetilde{T\partial^\alpha v} + \rho_\epsilon*\mu^\alpha,\, \forall\, \aiminabs{\alpha}=m.
\end{equation}
By the choice of $\rho$, $\rho_\epsilon*\mu^\alpha$ is supported outside of $\Omega$. Hence, restricting \eqref{eq07} to $\Omega$ implies that:
\begin{equation}\label{eq07extra1}
(\rho_\epsilon*(T\partial^\alpha\hat v))\big|_{\Omega} =\rho_\epsilon*(T\partial^\alpha v) \rightarrow T\partial^\alpha v, \text{ as }\epsilon\rightarrow 0,
\end{equation}
in $L^2(\Omega)$. Applying Lemma \ref{lem3.2extra} with $\mathcal{U}=\mathbb{R}^2, a=\lambda$ and $u=\partial^\alpha\hat v$, we obtain that as $\epsilon\rightarrow 0$,
\begin{equation}
T(\rho_\epsilon*\partial^\alpha\hat v) - \rho_\epsilon*T\partial^\alpha\hat v \rightarrow 0,\,\forall\, \aiminabs{\alpha}=m,
\end{equation}
in $L^2(\Omega)$, 
which, combined with \eqref{eq07extra1}, implies that
\begin{equation}\label{eq07extra2}
T(\rho_\epsilon*\partial^\alpha\hat v)\big|_{\Omega} \rightarrow T\partial^\alpha v, \text{ as }\epsilon\rightarrow 0,
\end{equation}
in $L^2(\Omega)$. If we set $\hat v_\epsilon=\rho_\epsilon*\hat v$, then as $\epsilon\rightarrow 0$, $\hat v_\epsilon\rightarrow\hat v$ in $H^m(\mathbb{R}^2)$, and
\begin{equation}\label{eq08}
\begin{cases}
\hat v_\epsilon\big|_\Omega \rightarrow v, \text{ in } H^m(\Omega);\\
T\partial^\alpha(\hat v_\epsilon\big|_\Omega)\rightarrow T\partial^\alpha v, \text{ in } L^2(\Omega),\, \forall\, \aiminabs{\alpha}=m,
\end{cases}
\end{equation}
which shows that $\hat v_\epsilon\big|_\Omega$ converges to $v$ in $\mathcal{X}_1^m(\Omega)$.
\end{proof}

Since we have to prove a density theorem involving the boundary values of the functions on $\partial\Omega$, we first need to show that the desired traces at the boundary make sense. We thus prove the following trace result.
\begin{prop}[A trace theorem]\label{prop3.3}
We assume that $\lambda=\lambda(x,y)$ satisfies the positive $0$-condition. If $\theta\in\mathcal{X}_1^0(\Omega)$, then the traces of $\theta$ are defined on all of $\partial\Omega$, i.e. the traces of $\theta$ are defined at $x=0,L_1$, and $y=0,L_2$, and they belong to the respective spaces $H_y^{-1}(0,L_2)$ and $H_x^{-1}(0,L_1)$. Furthermore the trace operators are linear continuous in the corresponding spaces, e.g., $\theta\in\mathcal{X}_1^0(\Omega)\rightarrow \theta|_{x=0}$ is continuous from $\mathcal{X}_1^0(\Omega)$ into $H_y^{-1}(0,L_2)$.
\end{prop}
\begin{proof}
Since $\theta\in L^2(\Omega)$,
we see that $\theta_y$ belongs to $L_x^2(0,L_1; H_y^{-1}(0,L_2))$, which implies that $\lambda \theta_x$ belongs to $L_x^2(0,L_1; H_y^{-1}(0,L_2))$ by observing that $T\theta=\lambda \theta_x+\theta_y\in L^2(\Omega)$. Using assumptions \eqref{asp3.1} and \eqref{asp3.2} when $m=0$ for $\lambda$, we obtain $\theta_x\in L_x^2(0,L_1; H_y^{-1}(0,L_2))$, which, in combination with $\theta\in L_x^2(0,L_1; L_y^2(0,L_2))$, shows that $\theta\in\mathcal{C}_x([0,L_1];H_y^{-1}(0,L_2))$. Hence, the traces of $\theta$ are well-defined at $x=0$ and $L_1$, and belong to $H_y^{-1}(0,L_2)$. The continuity of the corresponding mappings is easy. The proof for the traces at $y=0$ and $L_2$ is similar.
\end{proof}

We are now going to introduce density theorems involving the boundary values of the functions on $\partial\Omega$.
Here and throughout this article we denote by $\Gamma_1,\Gamma_2,\Gamma_3,\Gamma_4$ the boundaries $x=0,x=L_1,y=0,y=L_2$ respectively, and define ${\boldsymbol\Gamma}$ to be $\Gamma_1\cup\Gamma_3=\{x=0\}\cup\{y=0\}$.
We also write $\theta|_{\boldsymbol\Gamma}=0$ as a short notation for $\theta|_{x=0}=\theta|_{y=0}=0$, and we introduce the function spaces:
\begin{subequations}
\begin{align}
\mathcal{H}_{\boldsymbol\Gamma}^m(\Omega)&=\{\theta\in H^m(\Omega)\,:\,\partial^{\alpha}\theta\big|_{{\boldsymbol\Gamma}}=0,\, \forall\,\aiminabs{\alpha}\leq m-1 \},\label{eq3.e3a}\\
\mathcal{X}_{\boldsymbol\Gamma}^m(\Omega)&=\{ \theta\in H^m(\Omega)\,:\, T\theta=\lambda(x,y) \theta_x+\theta_y\in H^m(\Omega),\, \partial^{\alpha}\theta\big|_{{\boldsymbol\Gamma}}=0,\, \forall\,\aiminabs{\alpha}\leq m \}, \label{eq3.e3b}\\
\mathcal{V}_{\boldsymbol\Gamma}(\Omega) &= \{ \theta\in \mathcal{C}^\infty(\overline{\Omega})\,:\, \theta \text{ vanishes in a neighborhood of } {\boldsymbol\Gamma}\}.\label{eq3.e3c}
\end{align}
\end{subequations}
Note that when $m=0$, the space $\mathcal{H}_{\boldsymbol\Gamma}^0(\Omega)$ is the space $L^2(\Omega)$.

We first have the following characterizations for the space $\mathcal{X}_{\boldsymbol\Gamma}^m(\Omega)$.
\begin{prop}\label{prop3.4}
For all integer $m\geq 0$, we have
\begin{equation}\nonumber
\mathcal{X}_{\boldsymbol\Gamma}^m(\Omega) = \{ \theta\in \mathcal{H}_{\boldsymbol\Gamma}^m(\Omega)\,:\,T\theta\in \mathcal{H}_{\boldsymbol\Gamma}^m(\Omega),\,\partial^{\alpha}\theta\big|_{{\boldsymbol\Gamma}}=0,\, \forall\,\aiminabs{\alpha}=m\};
\end{equation}
for all integer $m\geq 1$, we have
\begin{equation}\label{eq3.1}
\mathcal{X}_{\boldsymbol\Gamma}^m(\Omega) = \{ \theta\in \mathcal{H}_{\boldsymbol\Gamma}^m(\Omega)\,:\,T\theta\in \mathcal{H}_{\boldsymbol\Gamma}^m(\Omega)\}.
\end{equation}
\end{prop}
\begin{proof}It is clear that the first statement holds for $m=0$, we thus only need to show the second statement. By definition of the spaces $\mathcal{H}_{\boldsymbol\Gamma}^m(\Omega)$ and $\mathcal{X}_{\boldsymbol\Gamma}^m(\Omega)$, we see that
$$\mathcal{X}_{\boldsymbol\Gamma}^m(\Omega) \subset \{ \theta\in \mathcal{H}_{\boldsymbol\Gamma}^m(\Omega)\,:\,T\theta\in \mathcal{H}_{\boldsymbol\Gamma}^m(\Omega)\}.$$
In order to prove the converse inclusion, let $\theta$ belongs to the right-hand side of \eqref{eq3.1}, then it is clear that we only need to show that $\theta$ satisfies the boundary conditions $\partial^{\alpha}\theta\big|_{{\boldsymbol\Gamma}}=0$ for all $\aiminabs{\alpha}= m$. Furthermore, we only need to show
that \begin{equation}
\partial_x^m\theta|_{\Gamma_1} = 0,\text{ and }\,\partial_y^m\theta|_{\Gamma_2} = 0,
\end{equation}
since the other boundary conditions involve the derivatives with respect to tangential directions. Since $T\theta\in \mathcal{H}_{\boldsymbol\Gamma}^m(\Omega)$, on $\Gamma_1=\{x=0\}$, we have
\begin{equation}
0=\partial_x^{m-1}(\lambda \theta_x + \theta_y)=\sum_{k=0}^{m-1}\partial_x^{m-k-1}(\lambda(x,y))\partial_x^k\theta_x + \partial_x^{m-1}\theta_y + \lambda\partial_x^m\theta,
\end{equation}
which, together with $\theta\in \mathcal{H}_{\boldsymbol\Gamma}^m(\Omega)$, implies that
$\lambda\partial_x^m\theta=0$. We thus have $\partial_x^m\theta=0$ on $\Gamma_1$. Similarly, we can also show that $\partial_y^m\theta=0$ on $\Gamma_2$. We thus completed the proof.

\end{proof}
As an immediate consequence of Proposition \ref{prop3.1}, we also find the following equivalent characterizations of the space $\mathcal{X}_{\boldsymbol\Gamma}^m(\Omega)$.
\begin{prop}\label{prop3.5}
\begin{equation}\nonumber
\begin{split}
\mathcal{X}_{\boldsymbol\Gamma}^m(\Omega)&=\{ \theta\in H^m(\Omega)\,:\, T\partial^\alpha\theta\in L^2(\Omega),\,\forall\, \aiminabs{\alpha}=m;\, \partial^{\alpha}\theta\big|_{{\boldsymbol\Gamma}}=0,\, \forall\,\aiminabs{\alpha}\leq m \} \\
&=\{ \theta\in \mathcal{H}_{\boldsymbol\Gamma}^m(\Omega)\,:\,T\partial^\alpha\theta\in L^2(\Omega),\,\partial^{\alpha}\theta\big|_{{\boldsymbol\Gamma}}=0,\, \forall\,\aiminabs{\alpha}=m\}.
\end{split}
\end{equation}
\end{prop}
Recall that ${\boldsymbol\Gamma}=\Gamma_1\cup\Gamma_3=\{x=0\}\cup\{y=0\}$, and then we state the density theorems:
\begin{thm}\label{thm3.1}
$\mathcal{V}_{\boldsymbol\Gamma}(\Omega)\cap\mathcal{H}_{\boldsymbol\Gamma}^m(\Omega)$ is dense in $\mathcal{H}_{\boldsymbol\Gamma}^m(\Omega)$.
\end{thm}
\begin{thm}\label{thm3.2}
Suppose that $\lambda=\lambda(x,y)$ satisfies the positive $m$-condition. Then we have
\begin{equation}\nonumber
\mathcal{V}_{\boldsymbol\Gamma}(\Omega)\cap\mathcal{X}_{\boldsymbol\Gamma}^m(\Omega) \text{ is dense in } \mathcal{X}_{\boldsymbol\Gamma}^m(\Omega).
\end{equation}
\end{thm}
\begin{rmk}\label{rmk3.1}
Theorem \ref{thm3.1} generalizes the classical density results, i.e. that $\mathcal{C}^\infty(\overline\Omega)$ is dense in $H^m(\Omega)$, and that $\mathcal{C}^\infty_c(\Omega)$ is dense in $H^m_0(\Omega)$, to the functions which vanish on part of the boundary $\partial\Omega$.

Theorem \ref{thm3.2} considers the density of the function space $\mathcal{X}_{\boldsymbol\Gamma}^m(\Omega)$ involving  functions like $T\theta$, hence this result is not included in Theorem \ref{thm3.1}.
\end{rmk}

The proof of Theorem \ref{thm3.1} is similar to or simpler than the proof of Theorem \ref{thm3.2}, we thus only prove Theorem \ref{thm3.2}.
To prove Theorem \ref{thm3.2}, we proceed similarly as in the proof of \cite[Theorem 1]{HT14a}
\begin{proof}[Proof of Theorem \ref{thm3.2}]
Let $\rho(x,y)$ be a mollifier such that $\rho(x,y)\geq 0, \int \rho \text{d}x\text{d}y=1$ and $\rho$ has compact support in $\{0<\frac{1}{2}x<y<2x\}$. 
For $\theta\in \mathcal{X}_{\boldsymbol\Gamma}^m(\Omega)$ and all $\alpha$ satisfying $|{\alpha}| \leq m$, we observe that,
\begin{equation}\label{eq122}
\partial^\alpha\tilde\theta = \widetilde{\partial^\alpha\theta} + \nu_1^\alpha,\qquad T\widetilde{\partial^\alpha\theta} = \widetilde{ T\partial^\alpha\theta} + \nu_2^\alpha,
\end{equation}
where $\nu_1^\alpha$ and $\nu_2^\alpha$ are measures supported by $\{x=L_1\}\cup\{y=L_2\}$. Therefore, we have
\begin{equation}\label{eq123}
T{\partial^\alpha\tilde\theta} = \widetilde{ T\partial^\alpha\theta} + \mu^\alpha,\qquad \forall\,|{\alpha}|\leq m,
\end{equation}
where $\mu^\alpha = T\nu_1^\alpha + \nu_2^\alpha$ is a measure also supported by $\{x=L_1\}\cup\{y=L_2\}$.

We now set $\tilde \theta_\epsilon=\rho_\epsilon*\tilde \theta$, and mollifying \eqref{eq123} with $\rho$ (see \cite{Hor65}) gives
\begin{equation}\label{eq124}
\rho_\epsilon*T{\partial^\alpha\tilde\theta} = \rho_\epsilon *\widetilde{ T\partial^\alpha\theta} + \rho_\epsilon* \mu^\alpha,\qquad \forall\,|{\alpha}|\leq m.
\end{equation}
By the choice of $\rho$, we have that $\rho_\epsilon*\mu^\alpha$ is supported in $\Omega^c$, and hence restricting \eqref{eq124} to $\Omega$ implies that
\begin{equation}\label{eq125}
(\rho_\epsilon*T{\partial^\alpha\tilde\theta})\big|_\Omega = ( \rho_\epsilon *\widetilde{ T\partial^\alpha\theta} )\big|_\Omega \rightarrow T\partial^\alpha\theta,\qquad \text{in }L^2(\Omega) \text{ as }\epsilon\rightarrow 0.
\end{equation}
Direct computation shows that
\begin{equation}\label{eq126}
T\partial^\alpha(\rho_\epsilon*\tilde\theta) - \rho_\epsilon*T\partial^\alpha\tilde\theta  =  \lambda\frac{\partial (\rho_\epsilon*\partial^\alpha\tilde\theta)}{\partial x} - \rho_\epsilon*\lambda\frac{\partial(\partial^\alpha\tilde\theta)}{\partial x} \rightarrow 0,\qquad \epsilon\rightarrow 0,
\end{equation}
where the convergence is in $L^2(\Omega)$  and achieved by applying Lemma \ref{lem3.2extra} with $\mathcal U=\mathbb R^2$, $a=\lambda$, and $u=\partial^\alpha\tilde\theta$. Combining \eqref{eq125} and \eqref{eq126} yields
\[
T\partial^\alpha(\rho_\epsilon*\tilde\theta) \big|_\Omega \rightarrow T\partial^\alpha\theta,\qquad \text{in }L^2(\Omega), \text{ as }\epsilon\rightarrow 0,
\]
that is for all $|{\alpha}|\leq m$,
\begin{equation}
T\partial^\alpha(\tilde\theta_\epsilon|_\Omega) \rightarrow T\partial^\alpha\theta,\qquad \text{in }L^2(\Omega), \text{ as }\epsilon\rightarrow 0.
\end{equation}
Similarly, by \eqref{eq122}, we have for all $|{\alpha}|\leq m$,
\begin{equation}
\partial^\alpha(\tilde\theta_\epsilon|_\Omega) = (\rho_\epsilon*\partial^\alpha\tilde\theta) |_\Omega = (\rho_\epsilon*\widetilde{ \partial^\alpha\theta })|_\Omega
\rightarrow \partial^\alpha\theta,\quad \text{in }L^2(\Omega)\text{ as } \epsilon\rightarrow 0,
\end{equation}
where we used the fact that the support of $\rho_\epsilon*\nu_1^\alpha$ is in $\Omega^c$. In conclusion, there holds
\begin{equation}
\begin{cases}
\tilde\theta_\epsilon|_\Omega \rightarrow \theta,\qquad\quad\qquad \text{ in }H^m(\Omega), \text{ as }\epsilon\rightarrow 0,\\
T\partial^\alpha(\tilde\theta_\epsilon|_\Omega) \rightarrow T\partial^\alpha\theta,\quad \text{in }L^2(\Omega), \text{ as }\epsilon\rightarrow 0,\quad\forall\,|{\alpha}|=m.
\end{cases}
\end{equation}
Finally, $\tilde \theta_\epsilon\big{|}_{\Omega}$ vanishes in a neighborhood of $\boldsymbol\Gamma$ since the support of $\tilde \theta_\epsilon\big{|}_{\Omega}$ is away from $\boldsymbol\Gamma$ by the choice of $\rho$. We thus completed the proof of Theorem \ref{thm3.2}.
\end{proof}

\begin{rmk}\label{rmk3.2}
Looking back carefully at the proof of Theorem \ref{thm3.2}, we see that Theorem \ref{thm3.2} is also valid if ${\boldsymbol\Gamma}=\Gamma_2\cup\Gamma_4$. Moreover, we say that $\lambda(x,y)$ satisfies the negative $m$-condition if $\lambda(x,y)$ satisfies \eqref{asp3.2} and the following condition:
\begin{equation}\tag{\ref{asp3.1}$'$}\label{asp3.1prime}
-c_1 \leq \lambda(x,y) \leq -c_0, 
\end{equation}
where $c_0,c_1$ are positive constants. Theorem \ref{thm3.2} is also true if ${\boldsymbol\Gamma}$ is $\Gamma_1\cup\Gamma_4$ or $\Gamma_2\cup\Gamma_3$, and $\lambda(x,y)$ satisfies the negative $m$-condition provided we choose properly the support of the mollifier.
\end{rmk}





\section{The time dependent shallow water equations operator}\label{sec-swe-operator}
In this section, we aim to study the semigroup property of the (modified) SWE operator (see below) with variable coefficients in the supercritical case on the Hilbert space $\mathcal{H}_{\boldsymbol\Gamma}^m(\Omega)$ (see \eqref{eq3.e3a}) with ${\boldsymbol\Gamma}=\Gamma_1\cup\Gamma_3$. We will successively consider the time-independent and the time-dependent cases. 
The linearized SWE operator that we consider reads
\begin{equation}\label{eq5.1}
\mathcal{A}(\widehat U) U = 
\begin{pmatrix}
\hat uu_x + \hat vu_y + g\phi_x  \\
\hat uv_x + \hat vv_y + g\phi_y  \\
\hat u\phi_x + \hat v\phi_y + \hat\phi(u_x+v_y) 
\end{pmatrix},
\end{equation}
where $\widehat U=(\hat u,\hat v,\hat\phi)^t, U=(u,v,\phi)^t$;
we set
\begin{equation}\nonumber
\mathcal{E}_1(\widehat U)=\begin{pmatrix}
\hat u&0&g\\
0&\hat u&0\\
\hat\phi&0&\hat u
\end{pmatrix},\hspace{6pt}
\mathcal{E}_2(\widehat U)=\begin{pmatrix}
\hat v&0&0\\
0&\hat v&g\\
0&\hat\phi&\hat v
\end{pmatrix}.
\end{equation}
Note that $\mathcal{E}_1$\footnote[2]{We sometimes write $\mathcal{E}_1=\mathcal{E}_1(\widehat U)$ for the sake of conciseness, etc.}, $\mathcal{E}_2$ admit a symmetrizer $S_0=\text{diag}(1,1,g/\hat\phi)$, i.e. $S_0\mathcal{E}_1, S_0\mathcal{E}_2$ are both symmetric. In order to take advantage of that, we consider the following modified SWE operator:
\begin{equation}\label{eq5.2}
\mathcal{A}^0(\widehat U) U = \mathcal{E}_1^0(\widehat U)U_x + \mathcal{E}_2^0(\widehat U) U_y,
\end{equation}
where
\begin{equation}\nonumber
\begin{split}
\mathcal{E}_1^0(\widehat U)&=S_0^{1/2}\mathcal{E}_1(\widehat U)S_0^{-1/2}=\begin{pmatrix}
\hat u&0&\sqrt{g\hat\phi}\\
0&\hat u&0\\
\sqrt{g\hat\phi}&0&\hat u
\end{pmatrix},\\
\mathcal{E}_2^0(\widehat U)&=S_0^{1/2}\mathcal{E}_2(\widehat U)S_0^{-1/2}=\begin{pmatrix}
\hat v&0&0\\
0&\hat v&\sqrt{g\hat\phi}\\
0&\sqrt{g\hat\phi}&\hat v
\end{pmatrix}.
\end{split}
\end{equation}
The reason why we choose the form \eqref{eq5.2} will become clear in the next section.

In the following, we assume that $m\geq 3$, the cases when $m=0,1,2$ are similar or simpler. Here, we only consider the generic case when $\widehat U$ does not vanish, and we first consider the time-independent case. We thus assume that $\widehat U$ only depends on the space variables $x,y$ and that
$\widehat U$ satisfies the positive $(m+1)$-condition (see \eqref{asp3.1}-\eqref{asp3.2}) introduced in Section \ref{sec-density}, i.e.
\begin{equation}\label{asp5.0}
\hat u,\hat v,\hat\phi\text{ satisfy the positive }(m+1)\text{-condition};
\end{equation}
the reason why we assume one more level of regularity on $\widehat U$ will be explained below. As we indicated before, we only study the supercritical case, and we thus assume that $\widehat U$ also satisfies the enhanced supercritical condition:
\begin{equation}\label{asp5.1}
\hat u^2 - g\hat\phi\geq c_2^2,\hspace{6pt} \hat v^2 - g\hat\phi \geq c_2^2,
\end{equation}
where $c_2$ is a positive constant.

\subsection{Boundary conditions}\label{subsec-boundary}
We aim to determine the boundary conditions which are suitable for the system
\begin{equation}\label{eq5.01}
\mathcal{A}^0(\widehat U) U = \mathcal{E}_1^0(\widehat U)U_x + \mathcal{E}_2^0(\widehat U) U_y = F,
\end{equation}
where $F\in\mathcal{H}^m_{\boldsymbol\Gamma}(\Omega)^3$. With assumption \eqref{asp5.1}, we see that $\mathcal{E}_1^0$ and $\mathcal{E}_2^0$ are both positive definite. Thus, it is natural to treat either the $x$- or $y$-direction as the time-like direction. Let us choose the $y$-direction, which means that we first need to specify the boundary conditions at $y=0$ (time-like initial conditions); choosing the $x$-direction would lead to the same result. Multiplying both sides of \eqref{eq5.01} by $(\mathcal{E}_2^{0})^{-1}$ gives
\begin{equation}\label{eq5.02}
 U_y + \mathcal{E}_2^{0}(\widehat U)^{-1} \mathcal{E}_1^0(\widehat U)U_x = \mathcal{E}_2^{0}(\widehat U)^{-1}F.
\end{equation}
We set $\kappa_0(\widehat U)=\sqrt{g(\hat u^2+\hat v^2-g\hat\phi)/\hat\phi}$, and we explicitly compute the eigenvalues of $(\mathcal{E}_2^{0})^{-1}\mathcal{E}_1^0$:
\begin{equation}
\lambda_1=\frac{\hat u\hat v+\hat\phi\kappa_0}{\hat v^2-g\hat\phi},\hspace{6pt}
\lambda_2=\frac{\hat u\hat v-\hat\phi\kappa_0}{\hat v^2-g\hat\phi},\hspace{6pt}
\lambda_3=\frac{\hat u}{\hat v}.
\end{equation}
Note that all the eigenvalues $\lambda_1,\lambda_2,\lambda_3$ of $ (\mathcal{E}_2^{0})^{-1} \mathcal{E}_1^0$ are positive under assumption \eqref{asp5.1}. Therefore, from the general hyperbolic theory (see Chapter 4 in \cite{BS07}), it is necessary and sufficient to specify the boundary conditions for $U$ at $x=0$ in order to solve  \eqref{eq5.02} in $U$. 

In conclusion, in order to solve \eqref{eq5.01} in $U$, we need to specify the boundary conditions for $U$ at $x=0$ and $y=0$. We then consider the homogeneous case and thus choose to specify the boundary conditions for $U$: 
\begin{equation}\label{eq5.3.1}
U=0, \text{ on }{\boldsymbol\Gamma}=\Gamma_{1}\cup\Gamma_3=\{x=0 \}\cup\{y=0 \}.
\end{equation}
As we will see in Lemma \ref{lem5.2} and Section \ref{sec-nonlinear-swe}, any sufficiently regular solution $U$ for \eqref{eq5.01} and for the \emph{nonlinear} equations \eqref{eq7.1} will satisfy the following compatibility boundary conditions:
\begin{equation}\label{eq5.3}
\begin{cases}
	\partial^{k}_{x}U=0, \text{ on } \Gamma_{1}=\{x=0 \},\,\forall\, 0\leq k\leq m, \\
	\partial^{k}_{y}U=0, \text{ on } \Gamma_{3}=\{y=0 \},\,\forall\, 0\leq k\leq m, 
\end{cases}	
\end{equation}
which, by differentiating with respect to the tangential direction, is equivalent to 
\begin{equation}\tag{\ref{eq5.3}$'$}\label{eq5.3prime}
\partial^\alpha U=0, \text{ on } {\boldsymbol\Gamma}=\Gamma_1\cup\Gamma_3,\,\forall\, \aiminabs{\alpha}\leq m.
\end{equation}
Hence in the following, we use the compatibility boundary conditions \eqref{eq5.3} rather than the boundary conditions \eqref{eq5.3.1} for the domain of the unbounded operator $A$ defined below.

We write $\mathcal{H}_{\boldsymbol\Gamma}^{k}=\mathcal{H}_{\boldsymbol\Gamma}^{k}(\Omega)^3$ for $k\geq 0$, in which the functions vanish on $\boldsymbol\Gamma$ (the part of the boundary $\partial\Omega$), and we endow the space $\mathcal{H}_{\boldsymbol\Gamma}^m$ with the Hilbert scalar product and norm of $H^m(\Omega)^3$:

$$\aimininner{U}{\overline U}_{\mathcal{H}_{\boldsymbol\Gamma}^m}=\sum_{\aiminabs{\alpha}\leq m}\aimininner{\partial^\alpha U}{\partial^\alpha \overline U}_{L^2(\Omega)} ,\hspace{6pt}\aiminnorm{U}_{\mathcal{H}_{\boldsymbol\Gamma}^m}=\{\aimininner{U}{ U}_{\mathcal{H}_{\boldsymbol\Gamma}^m}\}^{1/2}; $$
we then define the unbounded operator $A$ on $\mathcal{H}_{\boldsymbol\Gamma}^m$, by setting 
$AU=\mathcal{A}^0(\widehat U)U,\,\forall U\in\mathcal{D}(A)$ and 
\begin{equation}\nonumber
\mathcal{D}(A)=\{ U\in \mathcal{H}_{\boldsymbol\Gamma}^m\,:\, \mathcal{A}^0(\widehat U)U=\mathcal{E}_1^0(\widehat U)U_x + \mathcal{E}_2^0(\widehat U) U_y\in\mathcal{H}_{\boldsymbol\Gamma}^m\}.
\end{equation}

Note that the compatibility boundary conditions \eqref{eq5.3} are already taken into account in the domain $\mathcal{D}(A)$ (see also Propositions \ref{prop3.4}-\ref{prop3.5}) since $m\geq 3$. We also introduce the corresponding smooth function space  $\mathcal{V}(\Omega):=\mathcal{V}_{\boldsymbol\Gamma}(\Omega)^3$. Note that $\mathcal{V}(\Omega)$ is dense in $\mathcal{H}_{\boldsymbol\Gamma}^m$, which is a direct consequence of Theorem \ref{thm3.1}. We also have the following results.

\begin{lemma}\label{lem5.1}
We assume that $\widehat U$ satisfies the assumptions \eqref{asp5.0} and \eqref{asp5.1}. Then:
\begin{enumerate}[i)]
\item $\mathcal{V}(\Omega)\cap\mathcal{D}(A)$ is dense in $\mathcal{D}(A)$.
\item $\mathcal{D}(A)$ is dense in $\mathcal{H}_{\boldsymbol\Gamma}^m$.
\end{enumerate}
\end{lemma}
Lemma \ref{lem5.1} is proven below. In order to prove it, we need an equivalent characterization of the domain $\mathcal{D}(A)$, which will allow us to use the density results established in Section \ref{sec-density}. We introduce the notations $\kappa,\Xi,\mathcal{P}$ such that
\begin{equation}\label{eq5.5}
\begin{split}
\kappa(\widehat U) = \sqrt{ \hat{u}^2 + \hat{v}^2 - g\hat\phi },\hspace{6pt}
\mathcal{P}(\widehat U)^{-1} = \begin{pmatrix}
\hat v &-\hat u&\kappa        \\
\hat v  &-\hat u    &-\kappa\\
 \hat u   & \hat v & \sqrt{g\hat\phi}
\end{pmatrix},\hspace{6pt}
\Xi=\begin{pmatrix}\xi_1\\ \xi_2 \\ \xi_3 \end{pmatrix} = \mathcal{P}^{-1}U;
\end{split}
\end{equation}
then direct computations give
\begin{equation}\label{eq5.6}
\begin{split}
\mathcal{P}^t\mathcal{E}_1^0\mathcal{P}&=\text{diag}(\frac{\hat u\kappa + \sqrt{g\hat\phi}\hat v}{2(\hat u^2+\hat v^2)\kappa}, \frac{\hat u\kappa - \sqrt{g\hat\phi}\hat v}{2(\hat u^2+\hat v^2)\kappa}, \frac{\hat u}{\hat u^2 + \hat v^2}) =:\text{diag}(a_1,a_2,a_3), \\
\mathcal{P}^t\mathcal{E}_2^0\mathcal{P}&=\text{diag}(\frac{\hat v\kappa + \sqrt{g\hat\phi}\hat u}{2(\hat u^2+\hat v^2)\kappa}, \frac{\hat v\kappa - \sqrt{g\hat\phi}\hat u}{2(\hat u^2+\hat v^2)\kappa}, \frac{\hat v}{\hat u^2 + \hat v^2}) =:\text{diag}(b_1,b_2,b_3).
\end{split}
\end{equation}
We then rewrite the modified SWE operator as
\begin{equation}\label{eq5.7}
\begin{split}
\mathcal{P}^t\mathcal{A}^0(\widehat U)U&= \mathcal{P}^t\mathcal{E}_1^0(\mathcal{P}\Xi)_x + \mathcal{P}^t\mathcal{E}_2^0(\mathcal{P}\Xi)_y \\
&=\mathcal{P}^t\mathcal{E}_1^0\mathcal{P}\Xi_x + \mathcal{P}^t\mathcal{E}_2^0\mathcal{P}\Xi_y + \mathcal{P}^t\mathcal{E}_1^0\mathcal{P}_x\Xi + \mathcal{P}^t\mathcal{E}_2^0\mathcal{P}_y\Xi \\
&=\text{diag}(a_1,a_2,a_3)\Xi_x + \text{diag}(b_1,b_2,b_3)\Xi_y+ \mathcal{P}^t\mathcal{E}_1^0\mathcal{P}_x\Xi + \mathcal{P}^t\mathcal{E}_2^0\mathcal{P}_y\Xi. \\
\end{split}
\end{equation}

Direct computations also show that $a_i,b_i, i\in\{1,2,3\}$ are all positive away from $0$, and thus both $\mathcal{E}_1^0$ and $\mathcal{E}_2^0$ are symmetric and positive definite.
Using repeatedly Lemma \ref{lemb.1} and noting that $\widehat U$ belongs to $ H^{m+1}(\Omega)$, we see that $\mathcal{E}_1^0,\mathcal{E}_2^0,\kappa, \mathcal{P}, \mathcal{P}^{-1}, a_i,b_i, i\in\{1,2,3\}$ belong to $H^{m+1}(\Omega)$. Furthermore, the last two terms from the right-hand side of \eqref{eq5.7} both belong to $H^m(\Omega)^3$, and also to $\mathcal{H}_{\boldsymbol\Gamma}^m$ since $\mathcal{P}=\mathcal{P}(\widehat U)$ belongs to $H^{m+1}(\Omega)$, and that is the reason why we impose one more regularity on $\widehat U$. Therefore, saying that $\mathcal{A}^0(\widehat U)U$ belongs to $\mathcal{H}_{\boldsymbol\Gamma}^m$ is equivalent to saying that $a_i\xi_{i,x} + b_i\xi_{i,y}$ belongs to $\mathcal{H}_{\boldsymbol\Gamma}^m(\Omega)$ for all $i\in\{1,2,3\}$. Hence, the equivalent characterization of the domain $\mathcal{D}(A)$ is
\begin{equation}\nonumber
\mathcal{D}(A)=\{ U=\mathcal{P}\Xi\,:\,\Xi\in\mathcal{H}_{\boldsymbol\Gamma}^m,a_i\xi_{i,x} + b_i\xi_{i,y}\in\mathcal{H}_{\boldsymbol\Gamma}^m(\Omega),\forall\, i\in\{1,2,3\}\}.
\end{equation}
\begin{proof}[Proof of Lemma \ref{lem5.1}]

We remark that the statement $ii)$ directly follows from $i)$ since $\mathcal{V}(\Omega)$ is dense in $\mathcal{H}_{\boldsymbol\Gamma}^m$. We thus only need to prove $i)$.

Using the new characterization of $\mathcal{D}(A)$ and applying Theorem \ref{thm3.2} with $\lambda=a_i/b_i$ for all $i\in\{1,2,3\}$, we see that each component of $\Xi$ can be approximated by smooth functions which vanish in a neighborhood of ${\boldsymbol\Gamma}=\Gamma_1\cup\Gamma_3$. Then transforming back to the variable $U$, we obtain that $U$ can also be approximated by smooth functions in $\mathcal{V}(\Omega)$. The proof is complete.
\end{proof}

\subsection{Energy estimate for the operator $A$}In the following, we denote by $\aimininner{\cdot}{\cdot}$ the $L^2(\Omega)$-scalar product, and observe that if $\aiminnorm{\widehat U}_{H^{m}(\Omega)} \leq M$ with $m\geq 3$ (see Remark \ref{rmk5.0} below for the cases when $m=0,1,2$), then the $H^m$-norm of the functions $\mathcal{E}_1^0(\widehat U)$,
$\mathcal{E}_2^0(\widehat U)$ are bounded by some constant $C(M)$. Here and again in this section, the constant $C(M)$ may be different at different places, but it only depends on $M$. 
Then for $U$ smooth in $\mathcal{D}(A)$, we compute
\begin{equation}\label{eq5.9}
\begin{split}
\aimininner{-AU}{U}_{\mathcal{H}_{\boldsymbol\Gamma}^m} &= \sum_{\aiminabs{\alpha}\leq m}-\aimininner{\partial^\alpha(\mathcal{E}_1^0U_x + \mathcal{E}_2^0 U_y)}{\partial^\alpha U} \\
&=\sum_{\aiminabs{\alpha}\leq m} - \aimininner{\mathcal{E}_1^0(\partial^\alpha U)_x + \mathcal{E}_2^0 (\partial^\alpha U)_y}{\partial^\alpha U} \\
&\hspace{20pt} + \sum_{\aiminabs{\alpha}\leq m} -\big( \aimininner{[\partial^\alpha, \mathcal{E}_1^0]U_x}{\partial^\alpha U} + \aimininner{[\partial^\alpha, \mathcal{E}_2^0]U_y}{\partial^\alpha U} \big).
\end{split}
\end{equation}
Integrating by parts on the first summation at the right-hand side of \eqref{eq5.9} gives
\begin{equation}\label{eq5.10}
\begin{split}
\sum_{\aiminabs{\alpha}\leq m}\frac{1}{2}\bigg(-\int_0^{L_2}\aimininner{\mathcal{E}_1^0\partial^\alpha U}{\partial^\alpha U}\bigg|_{x=0}^{x=L_1}dy &- \int_0^{L_1}\aimininner{\mathcal{E}_2^0\partial^\alpha U}{\partial^\alpha U}\bigg|_{y=0}^{y=L_2}dx  \\
&+ \aimininner{\big(\mathcal{E}_{1,x}^0 + \mathcal{E}_{2,y}^0\big)\partial^\alpha U}{\partial^\alpha U}  \bigg).
\end{split}
\end{equation}
Using the compatibility boundary conditions \eqref{eq5.3prime} and that both $\mathcal{E}_1^0$ and $\mathcal{E}_2^0$ are positive definite to dispense with the boundary terms in \eqref{eq5.10}, we find that \eqref{eq5.10} is less than 
$ \sum_{\aiminabs{\alpha}\leq m}\frac{1}{2} \aimininner{\big(\mathcal{E}_{1,x}^0 + \mathcal{E}_{2,y}^0\big)\partial^\alpha U}{\partial^\alpha U}$,
which is dominated by
$\frac{1}{2}\big( \aiminnorm{\mathcal{E}_{1,x}^0}_{L^\infty} + \aiminnorm{\mathcal{E}_{2,y}^0}_{L^\infty} \big) \aiminnorm{U}_{\mathcal{H}_{\boldsymbol\Gamma}^m}^2$, which is finally bounded by $C(M)\aiminnorm{U}_{\mathcal{H}_{\boldsymbol\Gamma}^m}^2$ from the Sobolev embedding $H^2(\Omega)\subset L^\infty(\Omega)$, and $m\geq 3$. Applying Lemma \ref{lemb.1} \ref{en:item2}) with $k=m$ on the commutators from the right-hand side of \eqref{eq5.9}, we obtain that the second summation at the right-hand side of \eqref{eq5.9} is bounded by
\begin{equation}\label{eq5.11}
 \sum_{\aiminabs{\alpha}\leq m}C\bigg( \aiminnorm{\mathcal{E}_1^0}_{H^m} \aiminnorm{U_x}_{H^{\aiminabs{\alpha}-1}}\aiminnorm{\partial^\alpha U}_{L^2} + \aiminnorm{\mathcal{E}_2^0}_{H^m} \aiminnorm{U_y}_{H^{\aiminabs{\alpha}-1}}\aiminnorm{\partial^\alpha U}_{L^2} \bigg),
\end{equation}
which in turn is bounded by $C(M)\aiminnorm{U}_{\mathcal{H}_{\boldsymbol\Gamma}^m}^2$.

Gathering the estimates for \eqref{eq5.10} and \eqref{eq5.11}, \eqref{eq5.9} implies that
\begin{equation}
\aimininner{-AU}{U}_{\mathcal{H}_{\boldsymbol\Gamma}^m} \leq C_1(M)\aiminnorm{U}_{\mathcal{H}_{\boldsymbol\Gamma}^m}^2,
\end{equation}
which is
\begin{equation}\label{eq5.12}
\aimininner{AU}{U}_{\mathcal{H}_{\boldsymbol\Gamma}^m} \geq -C_1(M)\aiminnorm{U}_{\mathcal{H}_{\boldsymbol\Gamma}^m}^2.
\end{equation}
Thanks to Lemma \ref{lem5.1} $i$), we conclude that \eqref{eq5.12} holds for all $U$ in $\mathcal{D}(A)$.

\begin{rmk}\label{rmk5.0}
In the cases when $m=0,1,2$, we can easily check that the energy estimate \eqref{eq5.12} for the operator $A$ is also valid. Indeed, the estimate for the boundary terms is the same, and the estimate for the  commutators are simpler by direct calculation with the assumption that $\widehat U$ satisfies the positive $m$-condition ($m=0,1,2$ see \eqref{asp3.1}-\eqref{asp3.2}).
\end{rmk}

\subsection{The surjectivity of $\omega+A$}
We set $\omega_{0} = C_1(M)$, where $C_1(M)$ is the constant appearing in \eqref{eq5.12}, and we prove the following lemma.

\begin{lemma}\label{lem5.2}
Let $\omega$ be a real number which is greater than $\omega_0$. Then if $F$ belongs to $\mathcal{H}_{\boldsymbol\Gamma}^m$ with $m\geq 3$, the equation
\begin{equation}\label{eq5.15}
\mathcal{E}_1^0(\widehat U) U_x + \mathcal{E}_2^0(\widehat U) U_y + \omega U = F,
\end{equation}
associated with the following boundary conditions
\begin{equation}\label{eq5.16.1}
U=0, \text{ on }{\boldsymbol\Gamma}=\Gamma_{1}\cup\Gamma_{3}=\{y=0 \}\cup\{x=0 \},
\end{equation}
admits a unique solution $U$ in $\mathcal{D}(A)$.
\end{lemma}
\begin{proof}
Since $\mathcal{E}_2^0$ is nonsingular, multiplying by $(\mathcal{E}_2^0)^{-1}$ on both sides of \eqref{eq5.15} gives
\begin{equation}\label{eq5.15.1}
 U_y + \mathcal{E}_2^0(\widehat U)^{-1}\mathcal{E}_1^0(\widehat U) U_x + \mathcal{E}_2^0(\widehat U)^{-1} \omega U =  \mathcal{E}_2^0(\widehat U)^{-1}F.
\end{equation}
Let us treat again the $y$-direction as the time-like direction; then \eqref{eq5.15.1} becomes a 1-dimensional hyperbolic system.
We observe that the boundary ($x$-direction only) is a regular open subset in $\mathbb{R}$, and that the boundary conditions satisfy the uniform Lopatinskii condition (see \cite[Chapter 9]{BS07} or \cite[Chapter 7]{CP82}). Hence, the general results in \cite[Chapter 9]{BS07} (see also \cite[Chapter 7]{CP82}) guarantee the existence and uniqueness of a solution $U$ for \eqref{eq5.15.1} and \eqref{eq5.16.1}. 
Using \eqref{eq5.15.1} and the boundary condition \eqref{eq5.16.1} at $y=0$, we can conclude by induction that
$\partial^k_y U|_{y=0} = 0$ for all $0\leq k\leq m$. Similar results also holds for the $x$-direction. Therefore, the solution $U$ also satisfies the compatibility boundary conditions:
\begin{equation}\label{eq5.16}
\begin{cases}
	\partial^{k}_{x}U=0, \text{ on } \Gamma_{1}=\{x=0 \},\,\forall\, 0\leq k\leq m, \\
	\partial^{k}_{y}U=0, \text{ on } \Gamma_{3}=\{y=0 \},\,\forall\, 0\leq k\leq m,
\end{cases}	
\end{equation}
since $F$ belongs to $\mathcal{H}_{\boldsymbol\Gamma}^m$ (i.e. it satisfies \eqref{eq5.16} with $m$ replaced by $m-1$).
It remains to show that the solution $U$ actually belongs to $\mathcal{H}_{\boldsymbol\Gamma}^m$ if $F$ belongs to $\mathcal{H}_{\boldsymbol\Gamma}^m$. 

For all $0\leq \aiminabs{\alpha}\leq m$, we deduce from \eqref{eq5.15} that $\partial^\alpha U$ satisfies the following equations
\begin{equation}\label{eq5.17}
\mathcal{E}_1^0 (\partial^\alpha U)_x + \mathcal{E}_2^0 (\partial^\alpha U)_y + \omega \partial^\alpha U = \partial^\alpha F - [\partial^\alpha, \mathcal{E}_1^0]U_x - [\partial^\alpha, \mathcal{E}_2^0]U_y.
\end{equation}
Taking the $L^{2}(\Omega)$ scalar product of each side of \eqref{eq5.17} with $\partial^\alpha U$ and integrating by parts, we arrive at
\begin{equation}\label{eq5.18}
\begin{split}
\omega\aiminnorm{\partial^\alpha U}_{L^2(\Omega)}^2 + &\frac{1}{2}\bigg(\int_0^{L_2}\aimininner{\mathcal{E}_1^0\partial^\alpha U}{\partial^\alpha U}\bigg|_{x=0}^{x=L_1}dy + \int_0^{L_1}\aimininner{\mathcal{E}_2^0\partial^\alpha U}{\partial^\alpha U}\bigg|_{y=0}^{y=L_2}dx \bigg) \\
&=  \aimininner{\partial^\alpha F}{\partial^\alpha U} + \frac{1}{2}\aimininner{\big(\mathcal{E}_{1,x}^0 + \mathcal{E}_{2,y}^0\big)\partial^\alpha U}{\partial^\alpha U} \\
&\hspace{20pt}-\aimininner{[\partial^\alpha, \mathcal{E}_1^0]U_x}{\partial^\alpha U} - \aimininner{[\partial^\alpha, \mathcal{E}_2^0]U_y}{\partial^\alpha U}.
\end{split}
\end{equation}
The compatibility boundary conditions \eqref{eq5.16} and the fact that $\mathcal{E}_1^0$ and $\mathcal{E}_1^0$ are both positive definite imply that the boundary terms in the left-hand side of \eqref{eq5.18} are nonnegative, and thus the left-hand side of \eqref{eq5.18} is larger than $ \omega\aiminnorm{\partial^\alpha U}_{L^2(\Omega)}^2$. For the right-hand side of \eqref{eq5.18}, we use the Cauchy-Schwarz inequality to estimate the first term and the same arguments as for \eqref{eq5.10}-\eqref{eq5.12} to estimate the last three terms; then summing \eqref{eq5.18} for all $\aiminabs{\alpha}\leq m$ yields: 
\begin{equation}\label{eq5.19}
\begin{split}
\omega\aiminnorm{U}_{\mathcal{H}_{\boldsymbol\Gamma}^m}^2 &\leq \aiminnorm{F}_{\mathcal{H}_{\boldsymbol\Gamma}^m}\aiminnorm{U}_{\mathcal{H}_{\boldsymbol\Gamma}^m} + C_1(M)\aiminnorm{U}_{\mathcal{H}_{\boldsymbol\Gamma}^m}^2 \\
&\leq \aiminnorm{F}_{\mathcal{H}_{\boldsymbol\Gamma}^m}\aiminnorm{U}_{\mathcal{H}_{\boldsymbol\Gamma}^m} + \omega_0\aiminnorm{U}_{\mathcal{H}_{\boldsymbol\Gamma}^m}^2, \\
\end{split}
\end{equation}
with $\omega_0$ being a constant depending only on $M$. This implies that $U$ belongs to $\mathcal{H}_{\boldsymbol\Gamma}^m$ by the assumption $\omega>\omega_0$. Finally, $AU$ also belongs to  $\mathcal{H}_{\boldsymbol\Gamma}^m$ since $\partial^\alpha (AU)|_{\boldsymbol\Gamma} = 0,\,\forall \aiminabs{\alpha}\leq m-1$ and
\begin{equation}\nonumber
	AU=\mathcal{E}_1^0(\widehat U) U_x +\mathcal{E}_2^0(\widehat U) U_y =  F-\omega U .
\end{equation}
Therefore $U$ belongs to $\mathcal{D}(A)$, and the proof is complete.
\end{proof}

\subsection{Semigroup}We now set $B=\omega_0+A$, with $\mathcal{D}(B)=\mathcal{D}(A)$; then $(B,\mathcal{D}(B))$ is a positive operator on $\mathcal{H}_{\boldsymbol\Gamma}^m$ by virtue of \eqref{eq5.12}, and $\omega+B$ (=$\omega+\omega_{0}+A$) is surjective for all $\omega>0$ thanks to Lemma \ref{lem5.2}. Hence, Theorem \ref{thm2.1} (the Hille-Yosida theorem) implies that the operator $-B$ generates a contraction semigroup on $\mathcal{H}_{\boldsymbol\Gamma}^m$, and we then obtain the following result as a consequence of Theorem \ref{thm2.3} (Bounded Perturbation Theorem I).
\begin{thm}\label{thm5.1}
	The operator $(-A, \mathcal{D}(A))$ generates a quasi-contraction semigroup $(R(t))_{t\geq 0}$ on $\mathcal{H}_{\boldsymbol\Gamma}^m=\mathcal{H}_{\boldsymbol\Gamma}^m(\Omega)^3$ satisfying $\aiminnorm{R(t)} \leq e^{\omega_{0}t},\,\forall t\geq 0$.
\end{thm}
\begin{rmk}\label{rmk5.1}
The constant $\omega_0$ in Theorem \ref{thm5.1} only depends on the $H^m$-norm of $\widehat U$.
\end{rmk}

\subsection{Time-dependent modified SWE operator}\label{subsec-tSWE-operator}

In this subsection, we consider the case where $\widehat U$ also depends on the time variable $t$, and we impose the following assumptions on $\widehat U$:

\addtocounter{equation}{1}
\begin{enumerate}[(\theequation a)]
\item $\widehat U(t)$ satisfies the positive $(m+2)$-condition for all $t\in[0,T]$, i.e. $\widehat U(t)$ belongs to $H^{m+2}(\Omega)$ ($m\geq 3$) and it satisfies the condition \eqref{asp3.1} with $c_0,c_1$ independent of $t\in[0,T]$;\label{eq5.e1}
\item $\widehat U$ belongs to $\mathcal{C}([0,T];H^{m+1}(\Omega))$;\label{eq5.e2}
\item $\widehat U$ satisfies the supercritical condition \eqref{asp5.1} with $c_2$ independent of $t\in[0,T]$.\label{eq5.e3}
\end{enumerate}
Under these new assumptions, we see that the unbounded operator $-A$ defined in Subsection \ref{subsec-boundary} generates a strongly continuous semigroup with the same arguments as above, once we treat the time variable $t$ as a parameter. To be more precise, we 
define a family of unbounded operators $A(t)U$ on the Hilbert space $H$ with $A(t)U= \mathcal{A}^0(\widehat U(t))U,\,\forall U\in\mathcal{D}(A(t)|_H)$ and
\begin{equation}\nonumber
\mathcal{D}(A(t)|_H)=\{ U\in H\,:\, \mathcal{A}^0(\widehat U(t))U=\mathcal{E}_1^0(\widehat U(t))U_x + \mathcal{E}_2^0(\widehat U(t)) U_y\in H\},
\end{equation}
where $H=\mathcal{H}_{\boldsymbol\Gamma}^{k}(=\mathcal{H}_{\boldsymbol\Gamma}^{k}(\Omega)^3)$ and $k$ can be either $m-1$, $m$ or $m+1$.

Since the positive $(m+2)$-condition implies the positive $(m+1)$- and $m$-conditions, we thus obtain the following result as an immediate consequence of Theorem \ref{thm5.1}.
\begin{cor}
The operators $\{-A(t)\}_t$ generate quasi-contraction semigroups $(R_{t,1}(s))_{s\geq 0}$ on $\mathcal{H}_{\boldsymbol\Gamma}^{m-1}$, $(R_t(s))_{s\geq 0}$ on $\mathcal{H}_{\boldsymbol\Gamma}^m$, and $(R_{t,2}(s))_{s\geq 0}$ on $\mathcal{H}_{\boldsymbol\Gamma}^{m+1}$, and they satisfy
\begin{equation}\nonumber
\aiminnorm{R_{t,1}(s)} \leq e^{C(M_1)s},\hspace{6pt}
\aiminnorm{R_t(s)} \leq e^{C(M)s},\hspace{6pt}
\aiminnorm{R_{t,2}(s)} \leq e^{C(M_2)s},
\end{equation}
for all $s\geq 0$, where $M_1$ is the norm of $\widehat U$ in $\mathcal{C}([0,T];H^{m-1}(\Omega))$, $M$ is the norm of $\widehat U$ in $\mathcal{C}([0,T];H^m(\Omega))$, and $M_2$ is the norm of $\widehat U$ in $\mathcal{C}([0,T];H^{m+1}(\Omega))$.

Furthermore, with Remark \ref{rmk2.1}, it is clear that the family $\{-A(t)\}_{t\in [0,T]}$ is \emph{Kato-stable} (see Definition \ref{def-kato-stability}) in all these three spaces $\mathcal{H}_{\boldsymbol\Gamma}^{m-1}, \mathcal{H}_{\boldsymbol\Gamma}^m$ and $\mathcal{H}_{\boldsymbol\Gamma}^{m+1}$.
\end{cor}

By the definition of these spaces, the two embeddings
$$\mathcal{H}_{\boldsymbol\Gamma}^{m+1}\hookrightarrow \mathcal{H}_{\boldsymbol\Gamma}^m\hookrightarrow \mathcal{H}_{\boldsymbol\Gamma}^{m-1}$$
are dense and continuous.
Using Theorem \ref{thm2.2}, we obtain that $\mathcal{H}_{\boldsymbol\Gamma}^m$ (resp. $\mathcal{H}_{\boldsymbol\Gamma}^{m+1}$) is $-A(t)$ admissible for all $t\in[0,T]$ with respect to $\mathcal{H}_{\boldsymbol\Gamma}^{m-1}$ (resp. $\mathcal{H}_{\boldsymbol\Gamma}^m$). That $\mathcal{H}_{\boldsymbol\Gamma}^m\subset\mathcal{D}(A(t)|_{\mathcal{H}_{\boldsymbol\Gamma}^{m-1}})$ and $\mathcal{H}_{\boldsymbol\Gamma}^{m+1}\subset\mathcal{D}(A(t)|_{\mathcal{H}_{\boldsymbol\Gamma}^m})$ holds for all $t\in[0,T]$ is clear from the definition, and finally, that the mapping $t\mapsto -A(t)$ is continous in the $\mathcal{L}(\mathcal{H}_{\boldsymbol\Gamma}^m,\mathcal{H}_{\boldsymbol\Gamma}^{m-1})$-norm or $\mathcal{L}(\mathcal{H}_{\boldsymbol\Gamma}^{m+1},\mathcal{H}_{\boldsymbol\Gamma}^m)$-norm follows from the second assumption on $\widehat U$ (see (\theequation\ref{eq5.e2}) above). In conclusion, we find the following result.
\begin{lemma}\label{lem5.3}
Assume that $\widehat U$ satisfies the assumptions (\theequation\ref{eq5.e1})-(\theequation\ref{eq5.e3}) and $m\geq 3$.
Then the family $\{-A(t)\}_{t\in [0,T]}$ satisfies the \emph{Kato-condition} (see Definition \ref{def-kato-stability}) with $X=\mathcal{H}_{\boldsymbol\Gamma}^{m-1}$, $Y=\mathcal{H}_{\boldsymbol\Gamma}^m$ or $X=\mathcal{H}_{\boldsymbol\Gamma}^m$, $Y=\mathcal{H}_{\boldsymbol\Gamma}^{m+1}$.
\end{lemma}

\section{The linear shallow water system}\label{sec-linear}
In this section, we aim to study the well-posedness of the linear shallow water system in certain Sobolev spaces using the evolution semigroups technique. Keeping the notations introduced in Section \ref{sec-swe-operator}, the linear shallow water system reads in compact form
\begin{equation}\label{eq6.2}
U_t + \mathcal{E}_1(\widehat U) U_x + \mathcal{E}_2(\widehat U) U_y + \ell (U) = F,
\end{equation}
where $\ell (U)=(-fv,fu,0)^t$, and $f$ is the Coriolis parameter.
Note that $F$ which does not appear in the linearized shallow water system \eqref{eq1.1} is added here for mathematical generality and also for the study of the non-homogeneous boundary conditions or for the nonlinear case. Observe that the system \eqref{eq6.2} is Friedrichs symmetrizable (see Chapter 1 in \cite{BS07}) with symmetrizer $S_0=\text{diag}(1,1,g/\hat\phi)$, and in order to take advantage of that, we make as before a change of variables by setting $\widetilde U=S_0^{\frac{1}{2}}U$ and substitute into \eqref{eq6.2}; we obtain a new system for $\widetilde U$ which reads 
\begin{equation}\label{eq6.3}
\widetilde U_t + \mathcal{A}^0(\widehat U) \widetilde U + \mathcal{B}(\widehat U) \widetilde U = S_0^{\frac{1}{2}}F,
\end{equation}
where 
\begin{equation}
\begin{split}
\mathcal{A}^0(\widehat U) \widetilde U & = \mathcal{E}_1^0(\widehat U) \widetilde U_x + \mathcal{E}_2^0(\widehat U) \widetilde U_y, \\
\mathcal{B}(\widehat U) \widetilde U &=  S_0^{\frac{1}{2}}\big((S_0^{-\frac{1}{2}})_t + \mathcal{E}_1(\widehat U)(S_0^{-\frac{1}{2}})_x  +  \mathcal{E}_2(\widehat U)(S_{0}^{-\frac{1}{2}})_y \big) \widetilde U + \ell (\widetilde U).
\end{split}
\end{equation}

If we now assume that $\widehat U$ satisfies the conditions introduced in Subsection \ref{subsec-tSWE-operator}, then the family of operators $\{-A(t)\}_{t\in [0,T]}$ satisfies the \emph{Kato-condition} (see Lemma \ref{lem5.3}). If we further assume that $\widehat U$ belongs to $\mathcal{C}([0,T];H^{m+2}(\Omega))\cap\mathcal{C}^1([0,T];H^{m+1}(\Omega))$, then the operators $\{B(t)\}_{t\in [0,T]}$ defined by $B(t)\widetilde U=\mathcal{B}(\widehat U) \widetilde U$ are bounded operators on all the three spaces $\mathcal{H}_{\boldsymbol\Gamma}^{m-1},\mathcal{H}_{\boldsymbol\Gamma}^m,\mathcal{H}_{\boldsymbol\Gamma}^{m+1}$ by using the estimates in Lemma \ref{lemb.1}. Therefore, with Theorem \ref{thm2.5} (Bounded Perturbation Theorem II), the family of operators $\{-A(t)-B(t)\}_{t\in [0,T]}$ is a \emph{Kato-stable} family, and furthermore, we have
\begin{lemma}\label{lem6.1}
The family $\{-A(t)-B(t)\}_{t\in [0,T]}$ satisfies the \emph{Kato-condition} with $X=\mathcal{H}_{\boldsymbol\Gamma}^{m-1}$, $Y=\mathcal{H}_{\boldsymbol\Gamma}^m$ or $X=\mathcal{H}_{\boldsymbol\Gamma}^m$, $Y=\mathcal{H}_{\boldsymbol\Gamma}^{m+1}$.
\end{lemma}
Combining Theorem \ref{thm2.6} and Lemma \ref{lem6.1}, we obtain an evolution family on $\mathcal{H}_{\boldsymbol\Gamma}^{m-1}$ and another evolution family on $\mathcal{H}_{\boldsymbol\Gamma}^m$. From the uniqueness in Theorem \ref{thm2.6}, we see that these two evolution families coincide on $\mathcal{H}_{\boldsymbol\Gamma}^m$.
Then this unique evolution family satisfies $(E_1)-(E_3)$ with $X=\mathcal{H}_{\boldsymbol\Gamma}^{m-1}$ and satisfies $(E_4)-(E_5)$ with $Y=\mathcal{H}_{\boldsymbol\Gamma}^m$ (see Theorem \ref{thm2.6} and \ref{thm2.7}). Using Theorem \ref{thm2.7}, we obtain that the following system
\begin{equation}
\begin{cases}
\frac{d\widetilde U(t)}{dt} =-(A(t)+B(t))\widetilde U(t) + S_0^{\frac{1}{2}}F(t),\\
\widetilde U(0)=\widetilde U_0
\end{cases}
\end{equation}
admits a unique solution $\widetilde U=\widetilde U(t)\in\mathcal{C}([0,T];\mathcal{H}_{\boldsymbol\Gamma}^m)$ if $\widetilde U_0\in\mathcal{H}_{\boldsymbol\Gamma}^m$ and $F=F(t)\in \mathcal{C}([0,T];\mathcal{H}_{\boldsymbol\Gamma}^m)$. Transforming back to the original variables, we obtain the following:
\begin{thm}\label{thm6.1}
Let there be given $U_0\in\mathcal{H}_{\boldsymbol\Gamma}^m=\mathcal{H}_{\boldsymbol\Gamma}^m(\Omega)^3$ and $F=F(t)\in \mathcal{C}([0,T];\mathcal{H}_{\boldsymbol\Gamma}^m)$. We also assume that the $\widehat U(t)$ are given for all $t\in[0,T]$ such that
\begin{enumerate}[$(1)$]
\item $\widehat U(t)$ satisfies the positive $m+2$-condition for all $t\in[0,T]$, i.e. $\widehat U(t)$ belongs to $H^{m+2}(\Omega)$ ($m\geq 3$) and it satisfies the condition \eqref{asp3.1} with $c_0,c_1$ independent of $t\in[0,T]$,
\item $\widehat U$ belongs to $\mathcal{C}([0,T];H^{m+2}(\Omega))\cap\mathcal{C}^1([0,T];H^{m+1}(\Omega))$,
\item $\widehat U$ satisfies the supercritical condition \eqref{asp5.1} with $c_2$ independent of $t\in[0,T]$.
\end{enumerate}
Then the system \eqref{eq6.2} associated with the initial condition
$U(0)=U_0$ has a unique solution $U=U(t)$ which belongs to $\mathcal{C}([0,T];\mathcal{H}_{\boldsymbol\Gamma}^m)\cap\mathcal{C}^1([0,T];\mathcal{H}_{\boldsymbol\Gamma}^{m-1})$.
\end{thm}

\begin{rmk}\label{rmk6.0}
Using the system \eqref{eq6.2}, the solution $U$ in Theorem \ref{thm6.1} satisfies the compatibility boundary conditions \eqref{eq5.3} by the same argument as in Lemma \ref{lem5.2}.
\end{rmk}

\begin{rmk}\label{rmk6.1}
Notice that if $U_0\in\mathcal{C}^\infty(\overline\Omega)$\,\footnote[3]{For the domain $\Omega=(0,L_1)\times(0,L_2)$, we have that $\mathcal{C}^\infty(\overline\Omega)=\cap_{k=0}^\infty H^k(\Omega)$, see \cite[Chapter 1]{Gri85}.}, and $\widehat U, F\in\mathcal{C}^\infty([0,T];\mathcal{C}^\infty(\overline\Omega))$, then the solution $U$ provided by Theorem \ref{thm6.1} belongs to $\mathcal{C}([0,T];H^k(\Omega)^3)$ for all $k\geq 0$, which implies that $U$ belongs to $\mathcal{C}([0,T];\mathcal{C}^\infty(\overline\Omega))$, and then by using the system \eqref{eq6.2}, we conclude by induction that $U$ actually belongs to $\mathcal{C}^\infty([0,T];\mathcal{C}^\infty(\overline\Omega))$.
\end{rmk}

We lost two space derivatives from $\widehat U$ to the solution $U$ for the linear system \eqref{eq6.2} in Theorem \ref{thm6.1}, which is not sufficient for us to study the nonlinear case. In order to gain these two derivatives back, we need some additional a priori estimates. 

\subsection{A priori estimates}\label{sub-sec-priori}
With Remark \ref{rmk6.1}, we assume that $\widehat U, U_0, F,U$ are smooth functions  satisfying the following system 
\begin{equation}\label{eq6.4}
\begin{cases}
U_t + \mathcal{E}_1(\widehat U) U_x + \mathcal{E}_2(\widehat U) U_y + \ell(U) = F,\\
U(0)=U_0,\\
U|_{\boldsymbol\Gamma} = 0,
\end{cases}
\end{equation}
with $U_0, U(t), F(t)$ satisfying the compatibility boundary conditions \eqref{eq5.3} (i.e. \eqref{eq5.3prime}) for all $t\in[0,T]$, and ${\boldsymbol\Gamma}=\{x=0\}\cup\{y=0\}$.
In addition, $\widehat U$ is positive away from $0$ and satisfies the supercritical condition \eqref{asp5.1}, i.e.
\begin{equation}\label{asp6.1}
\begin{cases}
c_0 \leq \widehat U \leq c_1, \\
\hat u^2 - g\hat\phi\geq c_2^2,\, \hat v^2 - g\hat\phi \geq c_2^2,
\end{cases}
\end{equation}
where $c_0,c_1,c_2$ are positive constants.
We will first derive $L^2$ a priori estimates for the linear system \eqref{eq6.4} and then extend the $L^2$-estimates to $H^m$-estimates with $m\geq 3$. 
For the sake of simplicity, we write $\Omega_T=\Omega\times[0,T]$, and $L^\infty(H^k)=L^\infty(0,T;H^k(\Omega))$ for all $k=1,\cdots,m$ and $L^\infty(L^2)=L^\infty(0,T;L^2(\Omega))$. We assume that
\begin{equation}\label{asp6.2}
\aiminnorm{ \widehat U }_{L^\infty(H^m)} \leq M,\hspace{6pt}\aiminnorm{ \widehat U_t }_{L^\infty(H^{m-1})} \leq M.
\end{equation}

Multiplying \eqref{eq6.4}$_1$ by $S_0$ and taking the scalar product in $L^2(\Omega)$ with $U$ gives
\begin{equation}\label{eq6.5}
\aimininner{S_0U_t}{U} + \aimininner{S_0\mathcal{E}_1U_x+S_0\mathcal{E}_2U_y}{U} +\aimininner{S_0\ell(U)}{U} = \aimininner{S_0F}{U}.
\end{equation}
We now calculate
\begin{equation}\label{eq6.5_1}
\begin{split}
\aimininner{S_0U_t}{U}&=\frac{1}{2}\frac{d}{dt}\aimininner{S_0U}{U} - \frac{1}{2}\aimininner{S_{0,t}U}{U}, \\
\aimininner{S_0\ell(U)}{U} &= 0;
\end{split}
\end{equation}
and, using integration by parts, we find that
\begin{equation}\label{eq6.5_2}
\begin{split}
\aimininner{S_0\mathcal{E}_1U_x+S_0\mathcal{E}_2U_y}{U}&=\frac{1}{2}\int_0^{L_2}\aimininner{S_0\mathcal{E}_1 U}{U}\bigg|_{x=0}^{x=L_1}dy + \frac{1}{2}\int_0^{L_1}\aimininner{S_0\mathcal{E}_2 U}{ U}\bigg|_{y=0}^{y=L_2}dx  \\
&\hspace{20pt} - \frac{1}{2}\aimininner{\big((S_0\mathcal{E}_1)_x + (S_0\mathcal{E}_2)_y\big) U}{U} \\
&\geq - \frac{1}{2}\aimininner{\big((S_0\mathcal{E}_1)_x + (S_0\mathcal{E}_2)_y\big) U}{U},
\end{split}
\end{equation}
where the last inequality results from the boundary conditions \eqref{eq6.4}$_3$ and the fact that $S_0\mathcal{E}_1,S_0\mathcal{E}_2$ are both positive definite.

Finally, we obtain the following inequality by gathering the calculations \eqref{eq6.5_1}-\eqref{eq6.5_2}:
\begin{equation}\label{eq6.7}
\frac{d}{dt}\aimininner{S_0U}{U} \leq \aimininner{S_{0,t}U}{U} + \aimininner{\big((S_0\mathcal{E}_1)_x + (S_0\mathcal{E}_2)_y\big) U}{U} + 2\aimininner{S_0F}{U}.
\end{equation}
We set $I_0(t)=\aimininner{S_0U}{U}=\aiminnorm{S_0^{\frac{1}{2}}U}_{L^2(\Omega)}^2$, and then the first two terms in the right-hand side of \eqref{eq6.7} are bounded by
\begin{equation}\label{eq6.7_1}
\begin{split}
&\aiminnorm{\big(S_{0,t} + (S_0\mathcal{E}_1)_x + (S_0\mathcal{E}_2)_y \big)S_0^{-1}}_{L^\infty(\Omega_T)}\aimininner{S_0U}{U} \\
&\hspace{10pt}\leq C(\aiminnorm{\widehat U_t}_{L^\infty(\Omega_T)},\,\aiminnorm{\widehat U_x}_{L^\infty(\Omega_T)},\,\aiminnorm{\widehat U_y}_{L^\infty(\Omega_T)},\,\aiminnorm{\widehat U}_{L^\infty(\Omega_T)})I_0(t),
\end{split}
\end{equation}
which is dominated by $C(\aiminnorm{\widehat U_t}_{L^\infty(H^2)},\,\aiminnorm{\widehat U}_{L^\infty(H^3)})I_0(t)$ by using the Sobolev embedding $H^2(\Omega)\subset L^\infty(\Omega)$. Using the Cauchy-Schwarz inequality, we estimate the last term in the right-hand side of \eqref{eq6.7}:
\begin{equation}\label{eq6.7_2}
\begin{split}
2\aimininner{S_0F}{U} &\leq \aiminnorm{S_0^{\frac{1}{2}}F}_{L^2(\Omega)}^2 + \aiminnorm{S_0^{\frac{1}{2}}U}_{L^2(\Omega)}^2 \\
&\leq C( \aiminnorm{\widehat U}_{L^\infty(H^3)} )\aiminnorm{F(t)}_{L^2(\Omega)}^2 + I_0(t).
\end{split}
\end{equation}
Combining with \eqref{eq6.7_1} and \eqref{eq6.7_2}, \eqref{eq6.7} implies that
\begin{equation}\label{eq6.8}
\begin{split}
\frac{d}{dt}I_0(t) &\leq \big(C\big(\aiminnorm{\widehat U_t}_{L^\infty(H^2)},\,\aiminnorm{\widehat U}_{L^\infty(H^3)}\big) +1 \big)I_0(t) + C(\aiminnorm{\widehat U}_{L^\infty(H^3)} )\aiminnorm{F(t)}_{L^2(\Omega)}^2 \\
&\leq r_1(\widehat U)\big(I_0(t) + \aiminnorm{F(t)}_{L^2(\Omega)}^2\big),
\end{split}
\end{equation}
where the constant $r_1(\widehat U)=C\big(\aiminnorm{\widehat U_t}_{L^\infty(H^2)},\,\aiminnorm{\widehat U}_{L^\infty(H^3)}\big) +1$ only depends increasingly on $\aiminnorm{\widehat U_t}_{L^\infty(H^2)}$, $\aiminnorm{\widehat U}_{L^\infty(H^3)}$. We observe that 
$$r_1(\widehat U) = r_1(\aiminnorm{\widehat U_t}_{L^\infty(H^{m-1})},\,\aiminnorm{\widehat U}_{L^\infty(H^m)}) \leq r_1(M,M)=r_1(M),$$
with $m\geq 3$ by the assumption \eqref{asp6.2}, and we write $r_1=r_1(M)\geq 1$ for the sake of simplicity. Using Gronwall's lemma for \eqref{eq6.8}, we obtain
\begin{equation}\label{eq6.9}
\begin{split}
I_0(t) &\leq e^{r_1t}( I_0(0) + r_1\int_0^t \aiminnorm{F(s)}_{L^2(\Omega)}^2 ds) \\
&\leq e^{r_1t}r_1\cdot\big( \aiminnorm{U_0}_{L^2(\Omega)}^2 + \int_0^t \aiminnorm{F(s)}_{L^2(\Omega)}^2 ds \big).
\end{split}
\end{equation}
Noticing that
$$I_0(t)=\int_\Omega(u^2+v^2+\frac{g}{\hat\phi}\phi^2)dxdy \geq \text{min}(1,g/c_1)\aiminnorm{U(t)}_{L^2(\Omega)}^2,$$
and setting $r_2=1/\text{min}(1,g/c_1)$, \eqref{eq6.9} implies that
\begin{equation}\label{eq6.e1}
\aiminnorm{U(t)}_{L^2(\Omega)}^2\leq e^{r_1t}r_1r_2( \aiminnorm{U_0}_{L^2(\Omega)}^2 + \int_0^t \aiminnorm{F(s)}_{L^2(\Omega)}^2 ds ).
\end{equation}
Taking the $L^\infty$-norm of \eqref{eq6.e1} over $[0,T]$ immediately gives
\begin{equation}\label{eq6.e2}
\aiminnorm{U}^2_{L^\infty(L^2)}\leq C_0(M,T)\big( \aiminnorm{U_0}_{L^2(\Omega)}^2 + T\aiminnorm{F}^2_{L^\infty(L^2)}\big),
\end{equation}
where $C_0(M,T)=e^{r_1(M)T}r_1(M)r_2$, only depends on the bound of the $L^\infty(H^m)$-norm of $\widehat U$ and the $L^\infty(H^{m-1})$-norm of $\widehat U_t$.

We now turn to extending the $L^2$-estimate \eqref{eq6.e2} to $H^m$-estimate.
Applying $\partial^\alpha=\partial_x^{\alpha_1}\partial_y^{\alpha_2}$ with $\aiminabs{\alpha}=\alpha_1+\alpha_2 \leq m$ to \eqref{eq6.4} and recalling that $U$ satisfies the compatibility boundary conditions \eqref{eq5.3prime}, we obtain that $\partial^\alpha U$ satisfies the following equations
\begin{equation}\label{eq6.11}
\begin{cases}
(\partial^\alpha U)_t + \mathcal{E}_1(\widehat U) (\partial^\alpha U)_x + \mathcal{E}_2(\widehat U) (\partial^\alpha U)_y + \ell (\partial^\alpha U) = F_\alpha,\\
\partial^\alpha U(0)=\partial^\alpha U_0,\\
\partial^\alpha U|_{\boldsymbol\Gamma} = 0,
\end{cases}
\end{equation}
where
$ F_\alpha=\partial^\alpha F - [\partial^\alpha, \mathcal{E}_1]U_x -[\partial^\alpha, \mathcal{E}_2]U_y. $
Observing that \eqref{eq6.11} has the same form as \eqref{eq6.4}, therefore proceeding exactly as for \eqref{eq6.e2}, we find
\begin{equation}\label{eq6.12}
\aiminnorm{\partial^\alpha U}^2_{L^\infty(L^2)}\leq C_0(M,T)\big( \aiminnorm{\partial^\alpha U_0}_{L^2(\Omega)}^2 + T\aiminnorm{F_\alpha}^2_{L^\infty(L^2)}\big),
\end{equation}
where $C_0(M,T)$ is the same as in \eqref{eq6.e2}.

We now need to estimate $F_\alpha$. Lemma \ref{lemb.1} \ref{en:item2}) with $k=m$ on the commutators in $F_\alpha$ gives
\begin{equation}\label{eq6.e3}
\begin{split}
\aiminnorm{F_\alpha(t)}_{L^2(\Omega)}^2 &\leq \aiminnorm{\partial^\alpha F(t)}_{L^2(\Omega)}^2 + C\bigg(\aiminnorm{\mathcal{E}_1(\widehat U(t))}_{H^m(\Omega)}^2\aiminnorm{U_x(t)}_{H^{\aiminabs{\alpha}-1}(\Omega)}^2 \\
&\hspace{30pt}+ \aiminnorm{\mathcal{E}_2(\widehat U(t))}_{H^m(\Omega)}^2\aiminnorm{U_y(t)}_{H^{\aiminabs{\alpha}-1}(\Omega)}^2\bigg) \\
&\leq \aiminnorm{\partial^\alpha F(t)}_{L^2(\Omega)}^2 + C(M) \aiminnorm{U(t)}_{H^{\aiminabs{\alpha}}(\Omega)}^2,
\end{split}
\end{equation}
where $C(M)$ only depends on $M$-the bound of the $L^\infty(H^m)$-norm of $\widehat U$.

Summing \eqref{eq6.12} for all $\aiminabs{\alpha}\leq m$ and using the estimates \eqref{eq6.e3} for $F_\alpha$, we finally arrive at
\begin{equation}\label{eq6.13}
\aiminnorm{U}^2_{L^\infty(H^m)} \leq C_0(M,T)\big( \aiminnorm{ U_0}_{H^m(\Omega)}^2 + T\aiminnorm{F}^2_{L^\infty(H^m)} + TC(M) \aiminnorm{U}^2_{L^\infty(H^m)}\big),
\end{equation}
where the constants $C(M)$ may be different at different places, but they enjoy the same property, i.e. they only depend on the bound of the $L^\infty(H^m)$-norm of $\widehat U$ in an increasing way.

We choose $T$ small enough so that $C_0(M,T)TC(M)\leq 1/2$; with this choice of $T$, we are able to absorb the term $\aiminnorm{U}_{L^\infty(H^m)}$ in the right-hand side of \eqref{eq6.13} and we find that
\begin{equation}\label{eq6.14}
\aiminnorm{U}^2_{L^\infty(H^m)} \leq 2 C_0(M,T)\big( \aiminnorm{ U_0}_{H^m(\Omega)}^2 + T\aiminnorm{F}^2_{L^\infty(H^m)} \big),
\end{equation}
where $C_0(M,T)$ is the same as in \eqref{eq6.e2}.
We emphasize the fact that the choice of $T$ only depends on the bound $M$ of the $L^\infty(H^m)$-norm of $\widehat U$ and the $L^\infty(H^{m-1})$-norm of $\widehat U_t$.

Finally, we estimate the $L^\infty(H^{m-1})$-norm of $U_t$. We write \eqref{eq6.4}$_1$ as
\begin{equation}\label{eq6.4.3}
U_t  = F -\ell(U)- \mathcal{E}_1(\widehat U) U_x - \mathcal{E}_2(\widehat U) U_y,
\end{equation}
We first take $H^{m-1}(\Omega)$-norm of \eqref{eq6.4.3} and use Lemma \ref{lemb.1} $i)$ with $s=m-1$ and $d=2$ to estimate the last two terms in the right-hand side of \eqref{eq6.4.3}; then we take $L^\infty$-norm over $[0,T]$, and we find
\begin{equation}\label{eq6.4.4}
\begin{split}
\aiminnorm{U_t}_{L^\infty(H^{m-1})} &\leq \aiminnorm{F}_{L^\infty(H^{m-1})} 
+ C\aiminnorm{\mathcal{E}_1(\widehat U)}_{L^\infty(H^{m-1})} \aiminnorm{U_x}_{L^\infty(H^{m-1})} \\
&\hspace{6pt}+ f\aiminnorm{U}_{L^\infty(H^{m-1})}  + C\aiminnorm{\mathcal{E}_2(\widehat U)}_{L^\infty(H^{m-1})} \aiminnorm{U_y}_{L^\infty(H^{m-1})} \\
&\leq \aiminnorm{F}_{L^\infty(H^{m-1})} +C(\aiminnorm{\widehat U}_{L^\infty(H^m)},\,f)\aiminnorm{U}_{L^\infty(H^m)},\\
\end{split}
\end{equation}
where $f$ is the Coriolis parameter. The inequality \eqref{eq6.4.4} shows that
\begin{equation}\label{eq6.4.5}
\aiminnorm{U_t}^2_{L^\infty(H^{m-1})}\leq 2\aiminnorm{F}^2_{L^\infty(H^{m-1})} + 2C_1(M)\aiminnorm{U}^2_{L^\infty(H^m)},
\end{equation}
where $C_1(M)$ only depends on $M$-the bound of $\widehat U$ in $L^\infty(H^m)$. We also obtain the following $L^\infty(L^2)$-estimate
\begin{equation}\label{eq6.4.6}
\aiminnorm{U_t}^2_{L^\infty(L^2)}\leq 2\aiminnorm{F}^2_{L^\infty(L^2)} + 2C_1(M)\aiminnorm{U}^2_{L^\infty(H^1)}.
\end{equation}

\subsection{Improved regularity}With the $H^m$-estimates \eqref{eq6.14} and \eqref{eq6.4.5} at hand, we are now able to gain back the derivatives lost in Theorem \ref{thm6.1} by shrinking down the time $T$, and we prove the following theorem. 
\begin{thm}\label{thm6.2}
Let there be given 
\[
U_0\in H^m(\Omega),\hspace{6pt} F,\,\widehat U\in L^\infty(0,T; H^m(\Omega)),\hspace{6pt}\widehat U_t\in L^\infty(0,T;H^{m-1}(\Omega)), 
\]
and furthermore we also assume that $U_0, F(t)$ satisfy the compatibility boundary conditions \eqref{eq5.3} for $t\in[0,T]$, and that $\widehat U$ satisfies \eqref{asp6.1} and \eqref{asp6.2}.
Then there exists $T>0$ small enough depending only on the bound of the $L^\infty(H^m)$-norm of $\widehat U$ and the $L^\infty(H^{m-1})$-norm of $\widehat U_t$ such that the system \eqref{eq6.2} associated with the initial condition
$U(0)=U_0$ and the homogeneous boundary conditions \eqref{eq5.3.1} has a unique solution $U=U(t)$ such that
\[
U\in L^\infty(0,T; H^m(\Omega)),\hspace{6pt} U_t\in L^\infty(0,T; H^{m-1}(\Omega)),
\]
and the solution $U$ satisfies the compatibility and boundary conditions \eqref{eq5.3} and the estimates \eqref{eq6.14} and \eqref{eq6.4.5}.
\end{thm}
\begin{proof}
Let $\rho(x,y),\sigma(t)$ be mollifiers such that $\rho(x,y),\sigma(t)\geq 0, \int \rho dxdy=\int \sigma dt=1$ and $\rho$ has compact support in $\{0<\frac{1}{2}x<y<2x\}$. 
For a function $w$ defined on $\Omega$, $(\rho_\epsilon*\tilde w)|_\Omega$ stands for the restriction to $\Omega$ of $\rho_\epsilon*\tilde w$, where $\tilde w$ is the extension of $w$ by $0$ outside $\Omega$, and similar notations are also used for the functions defined in $\Omega_T$, or the vector functions (with the notation applied to each component of the vector functions).
We then set
\begin{equation}\nonumber
U_0^\epsilon = (\rho_\epsilon*\widetilde{U_0})|_\Omega,\hspace{6pt}F^\epsilon=((\rho\sigma)_\epsilon*\widetilde F)|_\Omega,\hspace{6pt}\widehat U^\epsilon=((\rho\sigma)_\epsilon*\widetilde{\widehat U})|_\Omega.
\end{equation}
Standard mollifier theory shows that $U_0^\epsilon,F^\epsilon,\widehat U^\epsilon$ converge to $U_0,F,\widehat U$ respectively as $\epsilon\rightarrow 0$ in the corresponding spaces \footnote[4]{See the details in a related situation in \cite{HT14b}}. Hence for $\epsilon$ small enough, we can assume that
\begin{equation}\nonumber
\begin{split}
&\aiminnorm{U_0^\epsilon}_{H^m(\Omega)}^2\leq 2\aiminnorm{U_0}_{H^m(\Omega)}^2,\hspace{6pt}\aiminnorm{F^\epsilon}_{L^\infty(H^m)}^2\leq 2\aiminnorm{F}_{L^2(H^m)}^2;\\
&\aiminnorm{\widehat U^\epsilon}_{L^\infty(H^m)}\leq 2M,\hspace{6pt}\aiminnorm{\widehat U_t^\epsilon}_{L^\infty(H^{m-1})}\leq 2M.
\end{split}
\end{equation}
In addition, with the choice of $\rho$, we have that the support of $U_0^\epsilon$ is away from ${\boldsymbol\Gamma}=\{x=0\}\cup\{y=0\}$, and so is the support of $F^\epsilon(t)$ for all $t\in[0,T]$. Therefore, $U_0^\epsilon,F^\epsilon(t)$ also satisfy \eqref{eq5.3} for all $t\in[0,T]$.
Then using Theorem \ref{thm6.1} and Remarks \ref{rmk6.0}-\ref{rmk6.1}, there exists a smooth solution $U^\epsilon$ for system \eqref{eq6.4} with $U_0,F,\widehat U$ replaced by $U_0^\epsilon, F^\epsilon,\widehat U^\epsilon$, and $U^\epsilon$ also satisfies the compatibility boundary conditions \eqref{eq5.3}. For $T>0$ small enough only depending on the bound of the $L^\infty(H^m)$-norm of $\widehat U$ and the $L^\infty(H^{m-1})$-norm of $\widehat U_t$, then the a priori estimates \eqref{eq6.14} gives that
\begin{equation}\label{eq6.15}
\begin{split}
\aiminnorm{U^\epsilon}_{L^\infty(H^m)}^2 &\leq 2 C_0(2M, T)\big( \aiminnorm{ U_0^\epsilon}_{H^m(\Omega)}^2 + T\aiminnorm{F^\epsilon}^2_{L^\infty( H^m)}\big)\\
&\leq 4 C_0(2M, T)\big( \aiminnorm{ U_0}_{H^m(\Omega)}^2 + T\aiminnorm{F}^2_{L^\infty( H^m)}\big).
\end{split}
\end{equation}
where $2M$ is the bound of the $L^\infty(H^m)$-norm of $\widehat U^\epsilon$ and the $L^\infty(H^{m-1})$-norm of $\widehat U_t^\epsilon$.

The inequality \eqref{eq6.15} gives a uniform bound on the sequence $\{U^\epsilon\}$, which implies that there exists a subsequence of $\{U^\epsilon\}$ converging weak-star in $L^\infty(0,T; H^m(\Omega))$. The next point is to prove that the sequence $\{U^\epsilon\}$ is Cauchy in $L^\infty(0,T; L^2(\Omega))$. For that purpose, we write
\begin{equation}
W^\epsilon = U^\epsilon - U^{\epsilon'},
\end{equation}
and subtracting the corresponding equations of form \eqref{eq6.4} satisfied by $U^\epsilon$ and $U^{\epsilon'}$, we obtain
\begin{equation}\label{eq6.16}
\begin{cases}
W^\epsilon_t + \mathcal{E}_1(\widehat U^\epsilon) W^\epsilon_x + \mathcal{E}_2(\widehat U^\epsilon) W^\epsilon_y + \ell(W^\epsilon) = \widehat F,\\
W^\epsilon(0)=U_0^\epsilon - U_0^{\epsilon'}, \\
W^\epsilon|_{\boldsymbol\Gamma} = 0,
\end{cases}
\end{equation}
where
$$ \widehat F = F^\epsilon-F^{\epsilon'} + \big(\mathcal{E}_1(\widehat U^{\epsilon'})-\mathcal{E}_1(\widehat U^\epsilon)\big)U_x^{\epsilon'}+\big(\mathcal{E}_2(\widehat U^{\epsilon'})-\mathcal{E}_2(\widehat U^\epsilon)\big)U_y^{\epsilon'}.$$
Noticing that \eqref{eq6.16} has the same form as \eqref{eq6.4}, therefore proceeding exactly as for \eqref{eq6.e2}, we obtain
\begin{equation}\label{eq6.16_1}
\aiminnorm{ W^\epsilon }_{L^\infty(L^2)}^2 \leq C_0(2M,T)\big(\aiminnorm{U_0^\epsilon - U_0^{\epsilon'}}_{L^2(\Omega)}^2 + T\aiminnorm{\widehat F}_{L^\infty(L^2)}^2 \big).
\end{equation}
Using the explicit expressions for $\mathcal{E}_1$ and $\mathcal{E}_2$, direct computation shows that
\begin{equation}\label{eq6.16_2}
\aiminnorm{\mathcal{E}_1(\widehat U^{\epsilon'})-\mathcal{E}_1(\widehat U^\epsilon)}_{L^\infty(L^2)}^2,\aiminnorm{\mathcal{E}_2(\widehat U^{\epsilon'})-\mathcal{E}_2(\widehat U^\epsilon)}_{L^\infty(L^2)}^2 \leq 3 \aiminnorm{\widehat U^\epsilon - \widehat U^{\epsilon'}}_{L^\infty(L^2)}^2.
\end{equation}
Therefore, combining the estimates in \eqref{eq6.16_1} and \eqref{eq6.16_2}, we obtain 
\begin{equation}\label{eq6.17}
\begin{split}
\aiminnorm{ U^\epsilon - U^{\epsilon'} }_{L^\infty(L^2)}^2 &\leq C_0(2M,T)\big(\aiminnorm{U_0^\epsilon - U_0^{\epsilon'}}_{L^2(\Omega)}^2 + T\aiminnorm{F^\epsilon - F^{\epsilon'}}_{L^\infty(L^2)}^2 \\
&\hspace{16pt}+3T \aiminnorm{\widehat U^\epsilon - \widehat U^{\epsilon'}}_{L^\infty(L^2)}^2( \aiminnorm{U_x^{\epsilon'}}_{L^\infty(\Omega_T)}^2 + \aiminnorm{U_x^{\epsilon'}}_{L^\infty(\Omega_T)}^2 )\big),
\end{split}
\end{equation}
which, with the use of the Sobolev embedding $H^2(\Omega)\subset L^\infty(\Omega)$ and noting that $m\geq 3$, is furthermore bounded by
\begin{equation}\label{eq6.18}
 C_0(2M,T)\big(\aiminnorm{U_0^\epsilon - U_0^{\epsilon'}}_{L^2(\Omega)}^2 + T\aiminnorm{F^\epsilon - F^{\epsilon'}}_{L^\infty(L^2)}^2 
+6T\aiminnorm{U^{\epsilon'}}_{L^\infty(H^m)}^2 \aiminnorm{\widehat U^\epsilon - \widehat U^{\epsilon'}}_{L^\infty(L^2)}^2  \big).
\end{equation}

Since $\{U_0^\epsilon\}$ is a Cauchy sequence in $L^2(\Omega)$, and $\{F^\epsilon\},\{\widehat U^\epsilon\}$ are Cauchy sequences in $L^\infty(L^2)$, and $\{U^{\epsilon}\}$ is uniformly bounded in the $L^\infty(H^m)$-norm by \eqref{eq6.15}, we obtain from \eqref{eq6.17}-\eqref{eq6.18} that $\{U^\epsilon\}$ is also a Cauchy sequence in $L^\infty(0,T;L^2(\Omega))$. Hence by $L^2-H^m$ interpolation, the sequence $\{U^\epsilon\}$ converges strongly in $L^\infty(0,T;H^{m-1}(\Omega))$ to a function $U$ which belongs to $L^\infty(0,T;H^{m}(\Omega))$. 
Using Proposition \ref{prop3.3}, we obtain that $U$ satisfies the compatibility boundary conditions \eqref{eq5.3} since $U^\epsilon$ satisfies \eqref{eq5.3}.

The a priori estimates \eqref{eq6.4.5} give a uniform bound on the sequence $\{U^\epsilon_t\}$, i.e.
\begin{equation}
\begin{split}
\aiminnorm{U^\epsilon_t}^2_{L^\infty(H^{m-1})}&\leq 2\aiminnorm{F^\epsilon}^2_{L^\infty(H^{m-1})} + 2C_1(2M)\aiminnorm{U^\epsilon}^2_{L^\infty(H^m)}\\
&\leq 4\aiminnorm{F}^2_{L^\infty(H^{m-1})} + 2C_1(2M)\aiminnorm{U^\epsilon}^2_{L^\infty(H^m)},
\end{split}
\end{equation}
since the sequence $\{U^{\epsilon}\}$ are uniformly bounded in the $L^\infty(H^m)$-norm.

Proceeding exactly as for \eqref{eq6.4.6}, we obtain that
\begin{equation}\label{eq6.19}
\aiminnorm{U^\epsilon_t - U^{\epsilon'}_t}^2_{L^\infty(L^2)}\leq 2\aiminnorm{\widehat F}^2_{L^\infty(L^2)} + 2C_1(2M)\aiminnorm{U^\epsilon - U^{\epsilon'}}^2_{L^\infty(H^1)},
\end{equation}
which implies that $\{U^\epsilon_t\}$ is also a Cauchy sequence in $L^\infty(L^2)$ by using the above estimates for $\widehat F$ and noting that $\{U^\epsilon\}$ is Cauchy in $L^\infty(H^{m-1})$ with $m\geq 3$. Therefore, by $L^2-H^{m-1}$ interpolation, we obtain that $\{U^\epsilon_t\}$ converges strongly in $L^\infty(H^{m-2})$ to a function $V$ which belongs to $L^\infty(H^{m-1})$.

Now passing to the limit, we obtain that $U$ solves \eqref{eq6.4}, and $U_t=V$ at least in the sense of distributions. Finally, proceeding exactly as in Subsection \ref{sub-sec-priori}, we see that the solution $U$ satisfies the estimates \eqref{eq6.14} and \eqref{eq6.4.5}; the uniqueness directly follows from the estimate \eqref{eq6.14}. We thus completed the proof.
\end{proof}
\begin{rmk}[Non-homogeneous boundary conditions]\label{rmk6.2}
Using Remark 9.1 in \cite{HT14a}, the existence of a solution for the linear system \eqref{eq6.2} associated with non-homogeneous boundary conditions can be obtained, we omit the details here; see \cite{RTT08b} for a similar situation.
\end{rmk}

\section{The fully nonlinear shallow water system}\label{sec-nonlinear-swe}
In this section, we aim to investigate the well-posedness for Eqs. \eqref{eq1.1} associated with suitable initial conditions and homogeneous boundary conditions, and we will make a remark about the case of non-homogeneous boundary conditions. Keeping the notations introduced in Section \ref{sec-swe-operator} and \ref{sec-linear}, the fully nonlinear shallow water system reads in compact form
\begin{equation}\label{eq7.1}
U_t + \mathcal{E}_1(U)U_x + \mathcal{E}_2(U)U_y + \ell(U) = 0.
\end{equation}

\subsection{Stationary solution}
We want to study system \eqref{eq7.1} near a stationary solution, and we start by constructing such a stationary solution $(u,v,\phi)=(u_s, v_s, \phi_s)$. These functions are independent of time and satisfy
\begin{equation}\label{eq7.3}
\mathcal{E}_1(U)U_x + \mathcal{E}_2(U)U_y + \ell(U) = 0.
\end{equation}

The existence of the general stationary solution $U_s$ to \eqref{eq7.3}, which satisfies the supercritical condition, is a 1-dimensional hyperbolic problem if we multiply by $\mathcal{E}_2^{-1}$ on both sides of \eqref{eq7.3} and treat the $y$-direction as the time-like direction as we already did in Lemma \ref{lem5.2}. The general results in \cite[Chapter 11]{BS07} guarantee the existence of the stationary solution $U_s$ if we specify suitable initial ($y$-direction) and boundary ($x$-direction) conditions. Actually, Subsection 2.1 in \cite{HPT11} provides an $y$-independent stationary solution to \eqref{eq7.3} satisfying the supercritical condition. But in what follows, we think of $U_s$ in a general form (i.e. $U_s$ depends both $x$ and $y$).

We thus choose our stationary solution $U_s=(u_{s},v_{s}, \phi_{s})^t$ of \eqref{eq7.3} satisfying a strong form of the supercritical condition, i.e.
\begin{equation}\label{cond7.1}
	\begin{cases}
		2c_0 \leq u_s,v_s,\phi_s \leq \frac{1}{2}c_1,\\
		u_s^2 - g\phi_s\geq 2c_2^2,\, v_s^2 - g\phi_s \geq 2c_2^2,\\
	\end{cases}
\end{equation}
where $c_0,c_1,c_2$ are given, positive constants which will play the same role as those in \eqref{asp6.1}.

We set $U = U_s + \widetilde{U}$.
Note that if we choose $\delta$ sufficiently small so that if $\aiminabs{\widetilde U}\leq \delta$ (i.e. $\aiminabs{\tilde u},\aiminabs{\tilde v}, \aiminabs{\tilde\phi} < \delta$), then $U$ satisfies relations similar to \eqref{cond7.1}, that is:
\begin{equation}\label{cond7.2}
	\begin{cases}
		c_0 \leq \tilde u+u_s,\tilde v+v_s,\tilde \phi+\phi_s \leq c_1,\\
		(\tilde u+u_s)^2 - g(\tilde \phi+\phi_s)\geq c_2^2,\, (\tilde v+v_s)^2 - g(\tilde \phi+\phi_s) \geq c_2^2.\\
	\end{cases}
\end{equation}
The relations \eqref{cond7.2} will guarantee that we remain in the supercritical case.

We then substitute $U = U_s + \widetilde{U}$ into \eqref{eq7.1}; we obtain a new system for $\widetilde U$, and dropping the tildes, our new system reads:
\begin{equation}\label{eq7.5}
U_t + \mathcal{E}_1(U+U_s)U_x + \mathcal{E}_2(U+U_s)U_y + \ell(U) = F^U,
\end{equation}
where 
\begin{equation}\label{eq7.6}
\begin{split}
 F^U&= -\mathcal{E}_1(U+U_s)U_{s,x}-\mathcal{E}_2(U+U_s)U_{s,y}-\ell(U_s)\\
&=\begin{pmatrix}
uu_{s,x}+vu_{s,y}\\
uv_{s,x}+vv_{s,y}\\
u\phi_{s,x}+v\phi_{s,y}+\phi(u_{s,x}+v_{s,y})
\end{pmatrix}.
\end{split}
\end{equation}
In order to have the last equality, we use the fact that $U_s$ is a stationary solution satisfying \eqref{eq7.3}.
We supplement \eqref{eq7.5} with the following initial and homogeneous boundary conditions:
\begin{equation}\label{eq7.7}
\text{I.C. } U(0)=U_0,\hspace{6pt}\text{B.C. }U|_{x=0,y=0}=0.
\end{equation}
Observe that we can rewrite \eqref{eq7.5} as
\begin{equation}\label{eq7.e1}
U_x =\mathcal{E}_1(U+U_s)^{-1}\big( F^U-U_t -\mathcal{E}_2(U+U_s)U_y - \ell(U) \big);
\end{equation}
at $x=0$, we immediately see that $U_x|_{x=0}=0$. Applying $\partial_x^k$ with $1\leq k\leq m-1$ to \eqref{eq7.e1}, we can conclude by induction that
$\partial_x^k U|_{x=0} = 0$, for all $k=0,\cdots,m$.
Similarly, we also obtain that $\partial_y^k U|_{y=0} =0$ for all $k=0,\cdots,m$ at $y=0$. Therefore, if $U$ satisfies \eqref{eq7.5} and \eqref{eq7.7}, then $U$ also satisfies the compatibility boundary conditions \eqref{eq5.3}, i.e.
\begin{equation}\label{eq7.e2}
\begin{cases}
	\partial^{k}_{x}U=0, \text{ on } \Gamma_{1}=\{x=0 \},\,\forall\, 0\leq k\leq m, \\
	\partial^{k}_{y}U=0, \text{ on } \Gamma_{3}=\{y=0 \},\,\forall\, 0\leq k\leq m.
\end{cases}	
\end{equation}

\subsection{Nonlinear shallow water system}In order to be able to solve the nonlinear system \eqref{eq7.5}-\eqref{eq7.7}, we require the initial and boundary conditions to be compatible. We thus assume that $U_0$ satisfies the compatibility boundary conditions \eqref{eq7.e2} (i.e. \eqref{eq5.3}). We are now on the stage to prove the following result.

\begin{thm}\label{thm7.1}
Let there be given the stationary solution $U_s\in H^{m+1}(\Omega)$ with $m\geq 3$, and two positive constants $M_0,M$ such that
$$M_0,M>0,\hspace{6pt} \sqrt{ M_0 }\in (0,\frac{\delta}{\nu_m}],\hspace{6pt} M = \sqrt{M_0} + \aiminnorm{U_s}_{H^m(\Omega)}, $$
where $\nu_m$ denotes the norm of the Sobolev embedding $H^m(\Omega)\subset L^\infty(\Omega)$. We are also given the initial condition $U_0\in H^m(\Omega)$ which satisfies \eqref{eq7.e2} and
\begin{equation}\label{asp7.1}
\aiminnorm{U_0}^2_{H^m(\Omega)} \leq \text{min}\bigg(M_0,\,\frac{M_0}{2C(U_s) + 4C_1(M)C_0(M,1)},\,  \frac{M_0}{ 4C_0(M,1) } \bigg),
\end{equation}
where $C_0(M,1)=e^{r_1(M)}r_1(M)r_2$ (resp. $C_1(M)$) is the constant appearing in \eqref{eq6.e2} (resp. \eqref{eq6.4.5}), and $C(U_s)$ only depends on the bound of the $H^{m+1}(\Omega)$-norm of $U_s$ (see \eqref{eq7.e6} below). 

Then there exists $T>0$ only depending on
the initial data $U_0$ and the stationary solution $U_s$
such that the system \eqref{eq7.5}-\eqref{eq7.7} admits a unique solution $U$ satisfying
\begin{equation}\nonumber
U\in L^\infty(0,T; H^m(\Omega)),\hspace{6pt} U_t\in L^\infty(0,T; H^{m-1}(\Omega)).
\end{equation}
\end{thm}
\begin{proof}
As a preliminary, we choose $T$ such that $T\leq 1$.

Considering the compatibility boundary conditions \eqref{eq7.e2}, the resolution of the nonlinear system \eqref{eq7.5}-\eqref{eq7.7} will be done using the following iterative scheme
\begin{equation}\label{eq7.9}
\begin{cases}
U_t^{k+1} + \mathcal{E}_1(U^k+U_s)U^{k+1}_x + \mathcal{E}_2(U^k+U_s)U^{k+1}_y + \ell (U^{k+1}) = F^{U^k},\\
U^{k+1}(0)=U_0,\\
U^{k+1}|_{\boldsymbol\Gamma} = 0.\\
\end{cases}
\end{equation}
We initiate our iteration scheme by setting $U^0=U_0$, and then construct the approximate solutions $U^k$ by induction.

If we assume that 
\begin{equation}\label{eq7.10}
\aiminnorm{U^k}^2_{L^\infty(0,T;H^m(\Omega))}\leq M_0,\hspace{6pt}\aiminnorm{U_t^k}^2_{L^\infty(0,T;H^{m-1}(\Omega))}\leq M_0, \hspace{6pt} U^k\text{ satisfies \eqref{eq7.e2}},
\end{equation}
then we have that $F^{U^k}$ also satisfies \eqref{eq7.e2}. Using the Sobolev embedding $H^m(\Omega)\subset L^\infty(\Omega)$, the $L^\infty$-norm of $U^k$ is controlled by $\delta$, which shows that $U^k$ satisfies the supercritical condition \eqref{cond7.2}, i.e.
\begin{equation}\nonumber
\begin{cases}
c_0 \leq u^k+u_s,v^k+v_s,\phi^k+\phi_s \leq c_1,\\
(u^k+u_s)^2 - g(\phi^k+\phi_s)\geq c_2^2,\,(v^k+v_s)^2 - g(\phi^k+\phi_s)\geq c_2^2;
\end{cases}
\end{equation}
furthermore, we have
$$\aiminnorm{U^k+U_s}_{L^\infty(H^m)} \leq \sqrt{M_0} + \aiminnorm{U_s}_{H^m(\Omega)}\leq M,\hspace{6pt}\aiminnorm{(U^k+U_s)_t}_{L^\infty(H^{m-1})} \leq \sqrt{M_0} \leq M. $$
Therefore, for $T>0$ small enough only depending on $M$,
applying Theorem \ref{thm6.2} to \eqref{eq7.9} with $\widehat U= U^k+U_s,F=F^{U^k}$, gives a solution $U^{k+1}$ which satisfies \eqref{eq7.e2}, and that
\begin{equation}\label{eq7.e3}
\begin{split}
\aiminnorm{U^{k+1}}^2_{L^\infty(H^m)} &\leq 2 C_0(M,T)\big( \aiminnorm{ U_0}_{H^m(\Omega)}^2 + T\aiminnorm{F^{U^k}}^2_{L^\infty(H^m)} \big) \\
&\leq (\text{using Lemma \ref{lemb.1} $i)$ with $s=m$ and $d=2$ for $F^{U^k}$})\\
&\leq 2 C_0(M,1)\big( \aiminnorm{ U_0}_{H^m(\Omega)}^2 + TC\aiminnorm{U_s}^2_{H^{m+1}(\Omega)}\aiminnorm{U^k}^2_{L^\infty(H^m)}\big), 
\end{split}
\end{equation}
and
\begin{equation}\label{eq7.e4}
\begin{split}
\aiminnorm{U_t^{k+1}}^2_{L^\infty(H^{m-1})} \leq 2\aiminnorm{F^{U^k}}^2_{L^\infty(H^{m-1})} + C_1(M)\aiminnorm{U^{k+1}}^2_{L^\infty(H^m)},
\end{split}
\end{equation}
where $C_0(M,T)$ (resp. $C_1(M)$) is the constant appearing in \eqref{eq6.e2} (resp. \eqref{eq6.4.5}),
and $M$ is the bound of the $L^\infty(H^m)$-norm of $\widehat U$ ($=U^k+U_s$) and the $L^\infty(H^{m-1})$-norm of $\widehat U_t$ ($=U_t^k$).

It follows from the explicit form \eqref{eq7.6} of $F^U$ that
\begin{equation}
(F^U)_t = F^{U_t}.
\end{equation}
In particular, $(F^{U^k})_t = F^{U_t^k}$. Using also that $U^k|_{t=0}=U_0$, we obtain with the mean value theorem that
\begin{equation}
\aiminabs{ F^{U^k}(t) } \leq \aiminabs{ F^{U^k}(t) - F^{U_0} } + \aiminabs{ F^{U_0} } 
\leq t \aiminabs{ F^{U_t^k}(t') }  + \aiminabs{ F^{U_0} },
\end{equation}
for all $t\in[0,T]$ and for some $t'\in(0,t)$, which implies that
\begin{equation}\label{eq7.e5}
\begin{split}
\aiminnorm{F^{U^k}}^2_{L^\infty(H^{m-1})} 
&\leq 2T^2\aiminnorm{ F^{U_t^k} }^2_{L^\infty(H^{m-1})} + 2\aiminnorm{F^{U_0}}^2_{L^\infty(H^{m-1})}\\
&\leq (\text{using Lemma \ref{lemb.1} $i)$ with $s=m-1$ and $d=2$})\\
&\leq 2T^2C\aiminnorm{U_s}^2_{H^{m}(\Omega)}\aiminnorm{U_t^k}^2_{L^\infty(H^{m-1})} + 2C\aiminnorm{U_s}^2_{H^{m}(\Omega)}\aiminnorm{U_0}^2_{H^{m-1}(\Omega)}.
\end{split}
\end{equation}
Gathering the estimates \eqref{eq7.e3}-\eqref{eq7.e5}, and using the assumption \eqref{eq7.10}, we finally arrive at
\begin{equation}\label{eq7.e6}
\begin{cases}
\aiminnorm{U^{k+1}}^2_{L^\infty(H^m)}\leq 2 C_0(M,1)\aiminnorm{ U_0}_{H^m(\Omega)}^2 + TC_0(M,1)M_0C(U_s),\\
\aiminnorm{U_t^{k+1}}^2_{L^\infty(H^{m-1})} \leq \big(C(U_s) + 2C_1(M)C_0(M,1)\big) \aiminnorm{ U_0}_{H^m(\Omega)}^2 \\
\hspace{100pt}+ TM_0C(U_s)\big(T+2C_1(M)C_0(M,1)\big),
\end{cases}
\end{equation}
where $C(U_s)$ only depends on the bound of the $H^{m+1}(\Omega)$-norm of $U_s$.
Note that the first two terms in the right-hand side of \eqref{eq7.e6} are less than $M_0/2$ by the assumption \eqref{asp7.1}, and both the second terms in the right hand side of \eqref{eq7.e6} approach $0$ when $T\rightarrow 0$; we thus can choose $T$ small enough again such that
\begin{equation}\label{eq7.11}
\aiminnorm{U^{k+1}}^2_{L^\infty(H^m)}\leq M_0,\hspace{6pt}\aiminnorm{U_t^{k+1}}^2_{L^\infty(H^{m-1})} \leq M_0,\hspace{6pt}U^{k+1}\text{ satisfies \eqref{eq7.e2}}.
\end{equation}
Now $U^{k+1}$ also satisfies \eqref{eq7.10}; hence we can continue our construction.
Let us emphasize that the choice of $T$ only depends on $M_0$, $M$, $U_s$ and is independent of $k$, therefore our iteration scheme can be conducted for all $k$, and we can construct the sequence $\{U^k\}$ as long as the starting point $U^0$ satisfies \eqref{eq7.10}, which holds true by the assumption \eqref{asp7.1}.

We now have an uniformly bounded sequence $\{U^k\}$ at hand, and the next point is to show that the sequence $\{U^k\}$ is Cauchy in $L^\infty(0,T;L^2(\Omega))$, which is almost achieved in the proof of Theorem \ref{thm6.2}. 
Let us write $W^{k+1}=U^{k+1}-U^k$; with $W^\epsilon,U^\epsilon,U^{\epsilon'},\widehat U^\epsilon,\widehat U^{\epsilon'},F^\epsilon,F^{\epsilon'}$ replaced by $W^{k+1},U^{k+1},U^{k},U^{k}+U_s,U^{k-1}+U_s,F^{U^{k}},F^{U^{k-1}}$ in \eqref{eq6.16}, proceeding exactly as for \eqref{eq6.17}-\eqref{eq6.18} and noticing that $W^{k+1}(0) = 0$, we obtain
\begin{equation}\label{eq7.12}
\begin{split}
\aiminnorm{ U^{k+1}-U^k }_{L^\infty(L^2)}^2 &\leq C_0(M,T)T\big( \aiminnorm{F^{U^k}-F^{U^{k-1}}}_{L^\infty(L^2)}^2 \\
&\hspace{20pt}+ 6 \aiminnorm{U^{k}}_{L^\infty(H^m)}^2 \aiminnorm{U^k-U^{k-1}}_{L^\infty(L^2)}^2\big).
\end{split}
\end{equation}
Using the explicit expression \eqref{eq7.6} for $F^U$ and the Sobolev embedding $H^m(\Omega)\subset L^\infty(\Omega)$, we estimate
\begin{equation}\label{eq7.13}
\begin{split}
\aiminnorm{F^{U^k}-F^{U^{k-1}}}_{L^\infty(L^2)}^2 &\leq C(\aiminnorm{U_s}_{L^\infty(\Omega)})\aiminnorm{U^k-U^{k-1}}_{L^\infty(L^2)}^2\\
&\leq C(\aiminnorm{U_s}_{H^m(\Omega)})\aiminnorm{U^k-U^{k-1}}_{L^\infty(L^2)}^2.
\end{split}
\end{equation}
Combining \eqref{eq7.12} and \eqref{eq7.13} and using the uniform boundedness of the $L^\infty(0,T; H^m(\Omega))$-norm of $U^k$, we obtain 
\begin{equation}\label{eq7.14}
\aiminnorm{ U^{k+1}-U^k }_{L^\infty(L^2)}^2 \leq C_0(M,T)T\big(C(\aiminnorm{U_s}_{H^m(\Omega)}) + 6M_0\big)\aiminnorm{U^k-U^{k-1}}_{L^\infty(L^2)}^2.
\end{equation}
Upon reducing $T$ again, we can assume that
\begin{equation}
C_0(M,T)T \big(C(\aiminnorm{U_s}_{H^m(\Omega)}) + 6M_0\big) \leq \frac{1}{4};
\end{equation}
then the inequality \eqref{eq7.14} implies that
\begin{equation}
\aiminnorm{ U^{k+1}-U^k }_{L^\infty(L^2)}\leq \frac{1}{2}\aiminnorm{U^k-U^{k-1}}_{L^\infty(L^2)}\leq\cdots\leq(\frac{1}{2})^k
\aiminnorm{U^1-U^0}_{L^\infty(L^2)}.
\end{equation}
Therefore \{$U^k$\} is Cauchy in $L^\infty(L^2)$; let $U$ be the limit of this sequence.
Note also that \{$U^k$\} is uniformly bounded in $L^\infty(H^m)$, so by $L^2-H^m$ interpolation,
the sequence \{$U^k$\} converges strongly in $L^\infty(H^{m-1})$ to $U\in L^\infty(H^m)$.
Similarly as for \eqref{eq6.19}, we can obtain that $\{U_t^k\}$ converges strongly in $L^\infty(H^{m-2})$ to a function $V\in L^\infty(H^{m-1})$.

Now passing to the limit in \eqref{eq7.9}, we obtain that $U$ solves \eqref{eq7.5} and $U_t=V$, and that $U$ belongs to $L^\infty(0,T;H^m(\Omega))$ and $U_t$ belongs to $L^\infty(0,T;H^{m-1}(\Omega))$. The uniqueness directly follows from \eqref{eq7.14}. This completes the proof.
\end{proof}
\begin{rmk}[Non-homogeneous boundary conditions]
With Remark \ref{rmk6.2}, the existence of a solution for the iterative scheme \eqref{eq7.9} associated with non-homogeneous boundary conditions can be obtained, and by passing to the limit, the nonlinear system \eqref{eq7.5} associated with non-homogeneous boundary conditions admits a unique solution; we omit the details here.
\end{rmk}

\begin{rmk}
After completing this article, we found that we can also use a finite difference method to prove the existence and uniqueness of the fully nonlinear SWE (i.e. Theorem \ref{thm7.1}) by observing that we have the energy estimates \eqref{eq5.19} for the corresponding boundary value problem (although slightly different), which is the one we only need for the finite difference method. We omit the details here.

The finite difference method has the advantages that we do not need the density theorems in Section \ref{sec-density} and the evolution semigroup technique. However, the evolution semigroup technique has its own advantage that it tells us how we lost two space derivatives for the well-posedness of the linear SWE (see Theorem \ref{thm6.1}), and that explains why we only have local well-posedness for the fully nonlinear SWE in some sense.
\end{rmk}

\appendix
\section{Preliminary results about semigroups and evolution families}\label{sec-semigroup}
This appendix collects some basic facts on the semigroups and evolution families and the characterization of their generators. The main references are the classical books by K. Yosida \cite{Yos80}, and by E. Hille and R.S. Phillips \cite{HP74}, and by A. Pazy \cite{Paz83} and the book by K.-J. Engel and R. Nagel \cite{EN00}.
\begin{defn}
A family $(S(t))_{t\geq 0}$ of bounded linear operators on a Banach space $X$ is called a strongly continuous (one-parameter) semigroup (or $\mathcal{C}_0$-semigroup) if it satisfies 
\begin{enumerate}[i)]
\item $S(0)=I,S(t+s)=S(t)S(s)$ for all $t,s\geq 0$;
\item $\xi_x : t\mapsto \xi_x(t):=S(t)x$ is continuous from $\mathbb{R}_+$ into $X$ for every $x\in X$.
\end{enumerate}
\end{defn}

\begin{prop}
For every strongly continuous semigroup $(S(t))_{t\geq 0}$, there exist constants $\omega\in\mathbb{R}$ and $M\geq 1$ such that
\begin{equation}\label{eqa.1}
\aiminnorm{S(t)}\leq Me^{\omega t}
\end{equation}
for all $t\geq 0$.
\end{prop}

\begin{defn}\label{defn2.2}
A strongly continuous semigroup is called \emph{quasi-contraction} if we can take $M=1$ in \eqref{eqa.1}, and  called \emph{bounded} if $\omega=0$, and called \emph{contraction} if $\omega=0$ and $M=1$ is possible. 
\end{defn}

\begin{defn}
The generator $A:\mathcal{D}(A)\subset X\mapsto X$ of a strongly continuous semigroup $(S(t))_{t\geq 0}$ on a Banach space $X$ is the operator
\begin{equation}\nonumber
Ax:=\dot{\xi}_x(0)=\lim_{h\downarrow 0}\frac{1}{h}(S(h)x -x)
\end{equation}
defined for every $x$ in its domain
\begin{equation}\nonumber
\mathcal{D}(A):=\{x\in X\,:\,t\mapsto\xi_x(t) \text{ is right differentiable in }t\text{ at }t=0\}.
\end{equation}
\end{defn}

Note that if $\xi_x(t)$ is right differentiable in $t$ at $t=0$, it is also differentiable at $t$, for any $t\geq 0$.

\begin{thm}[Hille-Yosida theorem]\label{thm2.1}
Let $(A,\mathcal{D}(A))$ be a positive operator on a Hilbert space $H$ such that $\omega+A$ is surjective for some $\omega>0$. Then $-A$ generates a contraction semigroup.
\end{thm}
We recall that on a Hilbert space, the linear operator $A$ is called \emph{positive} if $\aimininner{Ax}{x}\geq 0$ for all $x\in\mathcal{D}(A)$.



\begin{thm}[Uniqueness Theorem]\label{thm2.2}
Let $(A,\mathcal{D}(A))$ be a closed, densely defined operator on a Banach space $X$, and let $Y$ be a subspace of $X$ which is continuously embedded in $X$ (in symbols: $Y\hookrightarrow X$). The part of $A$ in $Y$ is the operator $A_|$ defined by $A_|y:=Ay$ with domain $\mathcal{D}(A_|)=\{y\in\mathcal{D}(A)\cap Y: Ay\in Y\}$.
Suppose that $(A,\mathcal{D}(A))$ generates a strongly continuous semigroup $(S(t))_{t\geq 0}$ on $X$, and $(A_|,\mathcal{D}(A_|))$ also generates a strongly continuous semigroup $(R(t))_{t\geq 0}$ on $Y$. Then 
\begin{equation}\label{eqa.2}
	S(t)y=R(t)y
\end{equation}
holds for all $y\in Y$ and $t\geq 0$. Furthermore, the subspace $Y$ is $A$-admissible, i.e. $Y$ is an invariant subspace of $S(t),\,t\geq 0$, and the restriction of $S(t)$ to $Y$, which is $R(t)$, is a strongly continuous semigroup on $Y$.
\end{thm}
\begin{proof}
The identity \eqref{eqa.2} immediately follows from the uniqueness of the following Cauchy problem
\begin{equation}\nonumber
\begin{cases}
	\dot{u}(t) = Au(t),\,\forall\, t\geq 0, \\
	u(0)=y\in Y;
\end{cases}
\end{equation}
and that $Y$ is $A$-admissible follows from \eqref{eqa.2} and the assumption.	

\end{proof}

\begin{thm}[Bounded Perturbation Theorem I]\label{thm2.3}
Let $(A,\mathcal{D}(A))$ be the infinitesimal generator of a strongly continuous semigroup $(S(t))_{t\geq 0}$ on a Banach space $X$ satisfying
$$\aiminnorm{S(t)}\leq M_0 e^{\omega t},\,\forall\, t\geq 0, $$
where $\omega\in \mathbb{R}, M_0\geq 1$. If $B\in \mathcal{L}(X)$, then
$C:=A+B$, with $\mathcal{D}(C):=\mathcal{D}(A)$
generates a strongly continuous semigroup $(R(t))_{t\geq 0}$ satisfying
\begin{equation}\nonumber
\aiminnorm{R(t)}\leq M_0 e^{(\omega + M_0\aiminnorm{B})t},\,\forall\, t\geq 0.
\end{equation}
\end{thm}

The following results are taken from \cite[Chapter 5]{Paz83}.
\begin{defn}
A two parameter family of bounded linear operators $W(t,s),0\leq s\leq t\leq T$, on a Banach space $X$ is called an \emph{evolution system} if the two following conditions are satisfied:
\begin{enumerate}[i)]
\item $W(s,s)=I, W(t,r)W(r,s) = W(t,s)$ for $0\leq s\leq r\leq t\leq T$;
\item $(t,s)\mapsto W(t,s)$ is strongly continuous for $0\leq s\leq t\leq T$.
\end{enumerate}
\end{defn}

Let $I$ be the interval $[0,T]$, and $(A(t),\mathcal{D}(A(t))_{t\in I}$ be the family of infinitesimal generators of strongly continuous semigroups $S_t(s), s\geq 0,$ on a Banach space $X$, and let $Y$ be a Banach space which is densely and continuously embedded into $X$. The following stability definition appeared in \cite{Kat70,Kat73}:
\begin{defn}[Kato-stability and Kato-condition]\label{def-kato-stability}
We say that the family $\{A(t)\}_{t\in I}$ is \emph{Kato-stable}, if there exist constants $M\geq 1$ and $\omega\in\mathbb{R}$ such that
\begin{equation}\nonumber
\aiminnorm{ \Pi_{j=1}^n e^{s_jA(t_j)}} \leq Me^{\omega\sum_{j=1}^n s_j},
\end{equation}
holds for every time-ordered sequence $(t_1,\cdots,t_n)$ in $I$ and $s_j\geq 0$;
and we say that the family $\{A(t)\}_{t\in I}$ satisfies the \emph{Kato-condition} on $I$ if the following conditions are satisfied:
\begin{enumerate}[i)]
\item $\{A(t)\}_{t\in I}$ is Kato-stable in $X$.
\item $Y$ is $A(t)$-admissible for all $t\in I$, and the family $\{A(t)|_Y\}$ of the part of $A(t)$ in $Y$ is Kato-stable in $Y$.
\item $Y\subset\mathcal{D}(A(t))$ holds for all $t\in I$, and for all $t\in I$, $A(t)$ is a bounded operator from $Y$ into $X$ and the mapping $t\mapsto A(t)$ is continuous in the $\mathcal{L}(Y,X)$ norm $\aiminnorm{\cdot}_{Y\rightarrow X}$.
\end{enumerate}
\end{defn}
\begin{rmk}\label{rmk2.1}
If for all $t\in I$, $A(t)$ is the infinitesimal generator of a \emph{quasi-contraction} (see Definition \ref{defn2.2}) semigroup $R_t(s)$ satisfying $\aiminnorm{R_t(s)}\leq e^{\omega s}$, then the family $\{A(t)\}_{t\in I}$ is clearly \emph{Kato-stable} with constants $M=1$ and $\omega$.
\end{rmk}

\begin{thm}[Bounded Perturbation Theorem II]\label{thm2.5}
Let $\{A(t)\}_{t\in I}$ be a \emph{Kato-stable} family of infinitesimal generators with constants $M$ and $\omega$. Let $\{B(t)\}_{t\in I}$ be bounded linear operators on $X$. If $\aiminnorm{B(t)}\leq K$ for all $t\in I$, then $\{A(t)+B(t)\}_{t\in I}$ is a \emph{Kato-stable} family of infinitesimal generators with constants $M$ and $\omega+KM$.
\end{thm}

\begin{thm}\label{thm2.6}
If the family $\{A(t)\}_{t\in I}$ satisfies the \emph{Kato-condition} (see Definition \ref{def-kato-stability}) then there exists a unique evolution system $W(t,s), 0\leq s\leq t\leq T$, in $X$ satisfying
\begin{equation}\tag{$E_1$}
\aiminnorm{ W(t,s) } \leq M e^{\omega(t-s)},\,\forall\, 0\leq s\leq t\leq T,
\end{equation}
\begin{equation}\tag{$E_2$}
\partial^+_t W(t,s)v\big|_{t=s} = A(s)v,\,\forall\, v\in Y, 0\leq s\leq T,
\end{equation}
\begin{equation}\tag{$E_3$}
\partial_s W(t,s)v = -W(t,s)A(s)v,\,\forall\, v\in Y, 0\leq s\leq t\leq T
\end{equation}
Here the derivaties  in ($E_2$) and ($E_3$) are in the strong sense in $X$ and  in ($E_2$) the derivative is from the right.
\end{thm}

\begin{thm}\label{thm2.7}
Let the family $\{A(t)\}_{t\in I}$ satisfies the \emph{Kato-condition} (see Definition \ref{def-kato-stability}) and let $W(t,s), 0\leq s\leq t\leq T$ be the evolution system given in Theorem \ref{thm2.6}. Suppose that $W(t,s)$ further satisfies both $(E_4)$ and $(E_5)$, where
\begin{equation}\tag{$E_4$}
W(t,s)Y\subset Y,\,\forall\, 0\leq s\leq t\leq T,
\end{equation}
\begin{equation}\tag{$E_5$}
(t,s)\mapsto W(t,s)v \text{ is continuous in }Y\text{ for all }v\in Y.
\end{equation}
Let $f\in\mathcal{C}([s,T]; Y)$, then for every $v\in Y, u(t)=W(t,s)v + \int_s^t W(t,r)f(r)dr$ is the unique $Y$-valued solution of the following initial value problem:
\begin{equation}\nonumber
\begin{cases}
\dfrac{du(t)}{dt}=A(t)u(t) + f(t),\,\forall\, 0\leq s\leq t\leq T, \\
u(s)=v;
\end{cases}
\end{equation}
and $u$ belongs to $\mathcal{C}([s,T]; Y) \cap \mathcal{C}^1([s,T]; X)$.
\end{thm}
\begin{rmk}
The definition of $Y$-valued solution only requires that $u$ belongs to $\mathcal{C}([s,T]; Y) \cap \mathcal{C}^1((s,T]; X)$ (see Definition 4.1 in \cite{Paz83}), but from the proof of Theorem 5.2 in \cite{Paz83}, we actually have $u\in\mathcal{C}([s,T]; Y) \cap \mathcal{C}^1([s,T]; X)$.
\end{rmk}

\section{Classical lemmas}\label{sec-classical-lemmas}
In this appendix, we collect some essential ingredients for Sobolev spaces (see e.g. Chapter 13 in \cite{Tay97} or Appendix C in \cite{BS07}).
\begin{lemma}\label{lemb.1}
Assume that ${\mathcal U}$ is a regular open set of $\mathbb{R}^d$, where $d$ is the dimension of the space. 
\begin{enumerate}[i$)$]
\item Consider $u$ and $v$ which both belong to $L^\infty({\mathcal U}) \cap H^s({\mathcal U})$ with $s>0$. Then their product also belongs to $H^s({\mathcal U})$ and there exists $C>0$ depending only on $s$ and ${\mathcal U}$ such that
$$
\aiminnorm{uv}_{H^s({\mathcal U})} \leq C( \aiminnorm{u}_{L^\infty({\mathcal U})} \aiminnorm{v}_{H^s({\mathcal U})} +\aiminnorm{v}_{L^\infty({\mathcal U})} \aiminnorm{u}_{H^s({\mathcal U})} ).
$$
If $s>d/2$, then the $L^\infty$ assumption automatically follows from the Sobolev embedding, and we have the following estimate:
$$
\aiminnorm{uv}_{H^s({\mathcal U})} \leq C \aiminnorm{u}_{H^s({\mathcal U})}\aiminnorm{v}_{H^s({\mathcal U})}.
$$
\item Let $\mathcal{F}$ be a $\mathcal{C}^\infty$ function on $\mathbb{R}$ such that $\mathcal{F}(0)=0$. Then there exists a continuous function $C:[0,+\infty)\mapsto[0,+\infty)$ such that for all $u\in H^s(\mathcal U)\cap L^\infty(\mathcal U)$ with $s\geq 0$:
$$ \aiminnorm{\mathcal{F}(u)}_{H^s(\mathcal U)} \leq C(\aiminnorm{u}_{L^\infty(\mathcal U)})\aiminnorm{u}_{H^s(\mathcal U)}.$$
If $s>d/2$, then the $L^\infty$ assumption automatically follows from the Sobolev embedding, and if furthermore we assume that $\mathcal U$ is bounded and that $u$ is positive away from $0$, i.e. $\aiminabs{u}\geq \epsilon_0$ for some positive $\epsilon_0$, then we have
$$\frac{1}{u}\in H^s(\mathcal U),$$
if we choose $\mathcal{F}$ to be a $\mathcal{C}^\infty$ function such that $\mathcal{F}(x)=0$ for $\aiminabs{x}\leq \epsilon_0/2$ and $\mathcal{F}(x)=1/x$ for $\aiminabs{x}\geq \epsilon_0$.

\item \label{en:item2} If $k$ is an integer greater than $d/2 + 1$ and $\alpha$ is a $d$-tuple of length $\aiminabs{\alpha}\in[1,k]$, there exists $C>0$ depending only on $k$ and ${\mathcal U}$ such that for all $a$ in $H^k({\mathcal U})$ and all $u\in H^{ \aiminabs{\alpha} -1}({\mathcal U})$, we have the following estimate:
$$
\aiminnorm{ [\partial^\alpha, a]u }_{L^2({\mathcal U})} \leq C \aiminnorm{a}_{H^k({\mathcal U})} \aiminnorm{u}_{H^{\aiminabs{\alpha}-1}({\mathcal U})}.
$$

\end{enumerate}
\end{lemma}
\section*{Acknowledgments}
This work was partially supported by the National Science Foundation under the grants NSF DMS-0906440 and DMS-1206438, and by the Research Fund of Indiana University.

\bibliographystyle{amsalpha}

\providecommand{\bysame}{\leavevmode\hbox to3em{\hrulefill}\thinspace}
\providecommand{\MR}{\relax\ifhmode\unskip\space\fi MR }
\providecommand{\MRhref}[2]{%
  \href{http://www.ams.org/mathscinet-getitem?mr=#1}{#2}
}
\providecommand{\href}[2]{#2}

\end{document}